\documentclass[10pt]{article}
\usepackage[utf8]{inputenc}
\usepackage{pgf,tikz}
\usepackage{caption}
\usepackage[outline]{contour}
\usepackage{subfigure}
\usepackage{authblk}
\usetikzlibrary{arrows.meta,bending,positioning}
\usepackage[left=4cm,right=4cm]{geometry}

\usepackage{amsmath, amssymb, amsfonts, mathtools, xcolor, amsthm, hyperref}
\theoremstyle{plain}
\newtheorem{theorem}{Theorem}

\newtheorem{proposition}[theorem]{Proposition}
\newtheorem{lemma}[theorem]{Lemma}

\newtheorem{problem}[theorem]{Problem}

\theoremstyle{definition}
\newtheorem{definition}[theorem]{Definition}

\title{On the range of two-distance graphs}

\author{Péter Ágoston\thanks{Supported by the Ministry of Innovation and Technology NRDI Office within the framework of the Artificial Intelligence National Laboratory (RRF-2.3.1-21-2022-00004), by the European Union, co-financed by the European Social Fund (EFOP-3.6.3-VEKOP-16-2017- 00002), by the Lend\"ulet program of the Hungarian Academy of Sciences (MTA), under the grant LP2017-19/2017 and by the Thematic Excellence Program TKP2021-NKTA-62 of the National Research, Development and Innovation Office.}}
\affil{ELTE E\"{o}tv\"{o}s Lor\'{a}nd University, Budapest, Hungary\\
HUN-REN Alfr\'{e}d R\'{e}nyi Institute of Mathematics, Budapest, Hungary}

\begin{document}

\maketitle

\begin{abstract}
The topic of this paper is related to the well-known notion of unit distance graphs. Take a graph with its edges coloured red and blue such that for some $d$ it can be mapped into the plane with all vertices going to distinct points, the red edges to segments of length $1$ and the blue edges to segments of length $d$. We define the range of this graph to be the set of such numbers $d$. It is easy to show that the range of any edge-bicoloured graph consists of finitely many intervals with algebraic endpoints, and we now prove that any such set with a finite positive upper and lower bound is the range of a suitable graph.
\end{abstract}

\section{Introduction}

\begin{definition}
We call a graph a unit distance graph (UDG), if it can be drawn to $\mathbb{R}^2$ so that all vertices go to distinct points and all neighbouring pairs of vertices have Euclidean distance $1$. We call such a drawing a unit distance representation (UDR) of the graph.
\end{definition}

From now on, we suppose that all graphs are finite and simple unless stated otherwise.

\begin{definition}
Call a graph an edge-bicoloured graph (EBG), if there is a fixed colouring of its edges with two colours.
\end{definition}

From now on, we will suppose that these colours are red and blue. Also, for any EBG $G$, call the set of its red and blue edges $E_r(G)$ and $E_b(G)$, respectively.

\begin{definition}
Call an EBG $G$ a $(1,d)$-graph for some $d\in\mathbb{R}_{\ge0}$, if the vertices of the graph can be represented in the plane by distinct points so that those connected with a red edge go to points with distance $1$ and those connected with a blue edge go to points with distance $d$. Call such a representation a $(1,d)$-representation of $G$.
\end{definition}

\begin{definition}
For an EBG $G$, define its range $ran(G)$ as the set of numbers for which $G$ is a $(1,d)$-graph. We may also define the range of an uncoloured graph as the union of the ranges of its edge-bicolourings (although the main result of this paper does not deal with this notion).
\end{definition}

We call a graph with or without an edge-bicolouring a two-distance graph if its range is not empty. Two-distance graphs have been studied in the past in several papers. {\cite{ei}}{\cite{p}}

When speaking about a $(1,d)$-representation of a graph, we often do not differentiate between vertices, edges and their images.

\begin{lemma}\label{lem:subgraph}
For any EBGs $H\subseteq G$ (where the colouring is also inherited from $G$ by $H$), $ran(G)\subseteq ran(H)$.
\end{lemma}

\begin{proof}
If $G$ has a $(1,d)$-representation for some $d$, then by taking the same mapping function for only the vertices of $H$, we trivially get a $(1,d)$-representation of $H$.
\end{proof}

\begin{lemma}\label{lem:union}
For EBGs $G_1$ and $G_2$, $ran(G_1\dot{\cup}G_2)=ran(G_1)\cap ran(G_2)$ (where $G_1\dot{\cup}G_2$ denotes the union of the two graphs, taken on two disjoint sets of vertices).
\end{lemma}

\begin{proof}
Suppose that for some $d$, both $G_1$ and $G_2$ have a $(1,d)$-representation. Then any translated version of the $(1,d)$-representation of $G_2$ is also a $(1,d)$-representation of $G_2$. Thus, we may fix a $(1,d)$-representation of $G_1$ and a $(1,d)$-representation of $G_2$ and translate the latter in a way that the sets of images of the vertices become disjoint. Thus, out of $\left\lvert V(G_1)\right\rvert\cdot\left\lvert V(G_2)\right\rvert+1$ translates, there is at least one, which is disjoint from the chosen $(1,d)$-representation of $G_2$. Thus, by drawing these two disjoint $(1,d)$-representations, we get a $(1,d)$-representation for $G=G_1\dot{\cup}G_2$.

Now suppose that for some $d$, at least one of $G_1$ and $G_2$ does not have a $(1,d)$-representation. We can assume WLOG that it is $G_1$. Then $G$ cannot have one either, since $G_1\subseteq G$, thus it would contradict Lemma \ref{lem:subgraph}.
\end{proof}

\begin{lemma}\label{lem:inverse}
If $G^*$ is an EBG obtained from another EBG $G$ by reversing its colouring, $ran(G^*)\cap\mathbb{R}_{>0}=\left\lbrace d\in\mathbb{R}_{>0}\vert\frac{1}{d}\in ran(G)\right\rbrace$.
\end{lemma}

\begin{proof}
Suppose that for some positive $d$, $G$ has a $(1,d)$-representation. Then we can obtain a $\left(1,\frac{1}{d}\right)$-representation of $G^*$, by scaling it down by $d$. Similarly, we can get a $\left(1,\frac{1}{d}\right)$-representation of $G$ from any $(1,d)$-representation of $G^*$.
\end{proof}

\begin{lemma}\label{lem:nonzero}
For an EBG $G$, $ran(G)=\left[0,+\infty\right)$ or $ran(G)\subseteq\mathbb{R}_{>0}$.
\end{lemma}

\begin{proof}
If the graph $G$ has at least one blue edge, then $0\notin ran(G)$, since otherwise in any $(1,0)$-representation, it would have coincident vertices, which contradicts the definition of $(1,d)$-representations. So $0\in ran(G)$ is only possible if $E_b(G)=\emptyset$ and $V(G)\cup E_r(G)$ forms a UDG in which case $ran(G)=\left[0,+\infty\right)$.
\end{proof}

Deciding whether a number $d$ is in the range of an EBG or not is $\mathbb{R}$-complete, since deciding whether a graph is a UDG  or not is $\mathbb{R}$-complete. {\cite{sch}}

$\chi\left(\mathbb{R}^2\right)$ denotes the minimal number of colours needed to colour $\mathbb{R}^2$ without a monochromatic pair of distance $1$. Finding $\chi\left(\mathbb{R}^2\right)$ is a famous problem {\cite{dg}} and Bukh conjectured {\cite{b}} that by also forbidding a transcendental distance, we get the same number. If true, this could make it interesting to find graphs whose range only contains a transcendental number.

\begin{definition}\label{def:semialgebraic}
Take the set of solutions $(x_1,...,x_n)$ to a finite sequence of polynomial equations and inequalities over some field $\mathbb{F}$ of the form $p(x_1,...,x_n)=0$ and $p(x_1,...,x_n)>0$. If a subset of $\mathbb{F}^n$ can be generated as the union of such sets, it is called a {\it semialgebraic set} over $\mathbb{F}^n$. $S\subseteq\mathbb{Q}$ is semialgebraic exactly if it is the union of finitely many intervals with algebraic endpoints (with some endpoints possibly being open, while other being possibly closed and some intervals possibly having length $0$). We will define the notion {\it extended semialgebraic set} for a subset of $\mathbb{R}$ with the same property: if it is the union of finitely many intervals with algebraic endpoints.
\end{definition}

\begin{proposition}\label{lem:laczkovich}
The range of an EBG $G$ is always an extended semialgebraic set.\footnote{Miklós Laczkovich, personal communication through Dömötör Pálvölgyi \url{https://dustingmixon.wordpress.com/2018/09/14/polymath16-eleventh-thread-chromatic-numbers-of-planar-sets/\#comment-6308}}
\end{proposition}

\begin{proof}
Let $V(G)=\left\lbrace v_1,v_2,...,v_{\left\lvert V(G)\right\rvert}\right\rbrace$. Also, take a Euclidean coordinate system on $\mathbb{R}^2$. Then for any $d$, a function $\varphi:V(G)\rightarrow\mathbb{R}^2$ for which $\varphi\left(v_i\right)=\left(x_i,y_i\right)$ ($1\le i\le\left\lvert V(G)\right\rvert$) determines a $(1,d)$-representation of $G$ exactly if the following polynomial inequalities all hold simultaneously:
\begin{gather*}
\left(x_i-x_j\right)^2+\left(y_i-y_j\right)^2=1\text{ if }\left(v_i,v_j\right)\in E_r(G), \\
\left(x_i-x_j\right)^2+\left(y_i-y_j\right)^2=d^2\text{ if }\left(v_i,v_j\right)\in E_b(G), \\
\left(x_i-x_j\right)^2+\left(y_i-y_j\right)^2>0\ (1\le i<j\le\left\lvert V(G)\right\rvert), \\
d\ge0
\end{gather*}

Thus, if we examine the set of points $\left(d,x_1,y_1,x_2,y_2,...,x_{\left\lvert V(G)\right\rvert},y_{\left\lvert V(G)\right\rvert}\right)$ in $\mathbb{R}^{2\left\lvert V(G)\right\rvert+1}$, they form an extended semialgebraic set $S$ there. From this, we can prove the statement using the Tarski--Seidenberg theorem \cite{t,s}.

\end{proof}

Our main result says this condition is tight if $ran(G)$ has positive lower and upper bounds:

\begin{theorem}\label{thm:main}
For a set $S\subseteq\mathbb{R}_{>0}$ with a positive lower and upper bound ($\lambda$ and $\upsilon$), there exists an EBG $G$ with $ran(G)=S$ if and only if $S$ is semialgebraic.
\end{theorem}

We first start with some preliminary algebraic statements, which reduce Theorem \ref{thm:main} to a slightly more natural statement (Proposition \ref{prop:main}) which we can prove using various graph tools. These are summarized in the beginning of Section \ref{sec:proof}.

\section{Preliminary algebraic statements}\label{sec:algebra}

\begin{definition}
Call a polynomial {\bf even}, if all of its coefficients with odd index are $0$. In other words, a polynomial $p$ is even, if it is an even function ($p(x)=p(-x)$.)
\end{definition}

\begin{definition}
Take a polynomial $p$, and $L\le U$, $\left\lbrace L,U\right\rbrace\subset\mathbb{R}_{\ge0}\cup\left\lbrace+\infty\right\rbrace$. Define:
\begin{align*}
S_0(p,L,U)&=\left\lbrace x\in\mathbb{R}_{>0}\vert\left(p(x)\ge0\right)\lor\left(x\le L\right)\lor\left(x\ge U\right)\right\rbrace\hbox{\ \ and}\\
S_1(p,L,U)&=\left\lbrace x\in\mathbb{R}_{>0}\vert\left(p(x)>0\right)\lor\left(x\le L\right)\lor\left(x\ge U\right)\right\rbrace\hbox{\ \ (Figure \ref{polinom11polinom12}).}
\end{align*}
\end{definition}

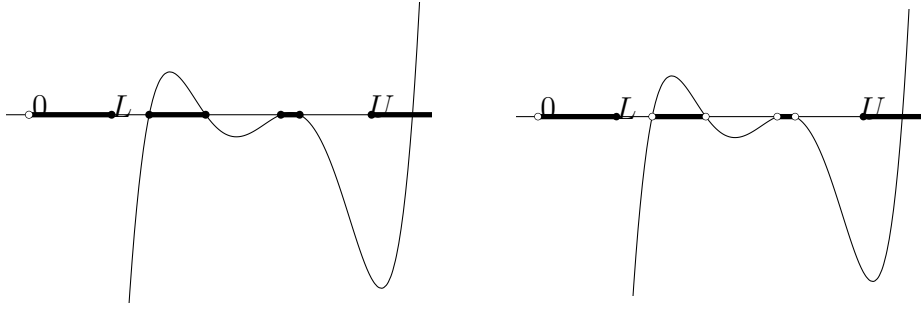
\begin{figure}
\begin{minipage}{.48\textwidth}
	\centering
	\usetikzlibrary{arrows}
\definecolor{ffffff}{rgb}{1,1,1}
\begin{tikzpicture}[line cap=round,line join=round]
\clip(-1.4,-2.5) rectangle (4.25,1.5);
\draw[smooth,samples=100,domain=-1:5] plot(\x,{1/60*(2*(\x)-1)*(2*(\x)-2.5)*(2*(\x)-4.5)*(2*(\x)-5)*(2*(\x)-8)});
\draw [line width=2pt] (0.5,0)-- (1.25,0);
\draw [line width=2pt] (2.25,0)-- (2.5,0);
\draw [domain=-3.81467:12.91754] plot(\x,{(-0-0*\x)/0.75});
\draw [line width=2pt] (-1.0984,0)-- (0,0);
\draw [line width=2pt,domain=3.455872535873332:12.917542058878636] plot(\x,{(-0-0*\x)/0.54413});
\begin{scriptsize}
\fill [color=black] (0.5,0) circle (1.5pt);
\fill [color=black] (1.25,0) circle (1.5pt);
\fill [color=black] (2.25,0) circle (1.5pt);
\fill [color=black] (2.5,0) circle (1.5pt);
\fill [color=black] (0,0) circle (1.5pt);
\draw[color=black] (0.13297,0.11773) node {\large{$L$}};
\fill [color=black] (-1.0984,0) circle (1.5pt);
\fill [color=ffffff] (-1.0984,0) circle (1.25pt);
\draw[color=black] (-0.95354,0.11773) node {\large{$0$}};
\fill [color=black] (3.45587,0) circle (1.5pt);
\draw[color=black] (3.60074,0.11773) node {\large{$U$}};
\end{scriptsize}
\end{tikzpicture}
\end{minipage}
\begin{minipage}{.48\textwidth}
	\centering
	\usetikzlibrary{arrows}
\definecolor{ffffff}{rgb}{1,1,1}
\begin{tikzpicture}[line cap=round,line join=round,scale=0.95]
\clip(-1.4,-2.5) rectangle (4.25,1.5);
\draw[smooth,samples=100,domain=-1:5] plot(\x,{1/60*(2*(\x)-1)*(2*(\x)-2.5)*(2*(\x)-4.5)*(2*(\x)-5)*(2*(\x)-8)});
\draw [line width=2pt] (0.5,0)-- (1.25,0);
\draw [line width=2pt] (2.25,0)-- (2.5,0);
\draw [domain=-3.81467:12.91754] plot(\x,{(-0-0*\x)/0.75});
\draw [line width=2pt] (-1.0984,0)-- (0,0);
\draw [line width=2pt,domain=3.455872535873332:12.917542058878636] plot(\x,{(-0-0*\x)/0.54413});
\begin{scriptsize}
\fill [color=black] (0.5,0) circle (1.5pt);
\fill [color=ffffff] (0.5,0) circle (1.25pt);
\fill [color=black] (1.25,0) circle (1.5pt);
\fill [color=ffffff] (1.25,0) circle (1.25pt);
\fill [color=black] (2.25,0) circle (1.5pt);
\fill [color=ffffff] (2.25,0) circle (1.25pt);
\fill [color=black] (2.5,0) circle (1.5pt);
\fill [color=ffffff] (2.5,0) circle (1.25pt);
\fill [color=black] (0,0) circle (1.5pt);
\draw[color=black] (0.13297,0.11773) node {\large{$L$}};
\fill [color=black] (-1.0984,0) circle (1.5pt);
\fill [color=ffffff] (-1.0984,0) circle (1.25pt);
\draw[color=black] (-0.95354,0.11773) node {\large{$0$}};
\fill [color=black] (3.45587,0) circle (1.5pt);
\draw[color=black] (3.60074,0.11773) node {\large{$U$}};
\end{scriptsize}
\end{tikzpicture}
\end{minipage}
\caption{A polynomial $p(x)$ with $S_0\left(p,L,U\right)$ (left) and $S_1\left(p,L,U\right)$ (right) denoted by bold}
\label{polinom11polinom12}
\end{figure}

\begin{proposition}\label{prop:algebra}
Take a semialgebraic set $\sigma\subseteq\left[\lambda,\upsilon\right]$. For some $n\in\mathbb{N}$ there exist even polynomials $p_1,...,p_{n+1}$ with integer coefficients and a negative leading coefficient, numbers $L_1,...,L_n,U_1,...,U_n\in\mathbb{Q}_{>0}$ and numbers $\zeta_1,...,\zeta_{n+1}\in\left\lbrace0,1\right\rbrace$ so that\\$\sigma=\left(\bigcap\limits_{i=1}^n{S_{\zeta_i}\left(p_i,L_i,U_i\right)}\right)\cap S_{\zeta_{n+1}}\left(p_{n+1},0,+\infty\right)$.
\end{proposition}

\begin{proof}
First, we will need the following lemma.

\begin{lemma}\label{lem:evenpolynomial}
For every algebraic number $\alpha$, there exists an even polynomial $\omega_{\alpha}(x)$ with integer coefficientes and a positive leading coefficient, which has $\alpha$ as a simple root.
\end{lemma}

\begin{proof}
The minimal polynomial of $\alpha$ (call it $\mu_{\alpha}$) has $\alpha$ as a simple root. If it does not have $-\alpha$ as a root, then we can take $\omega_{\alpha}(x)=\mu_{\alpha}(x)\cdot\mu_{\alpha}(-x)$, whose multiset of roots is exactly the union of the multiset of the roots of $\mu_{\alpha}$ and the multiset of the additive inverses of the roots of $\mu_{\alpha}$. Thus, if $\mu_{\alpha}(x)$ had $\alpha$ as a simple root and did not have $-\alpha$ as a root at all, then $\omega_{\alpha}(x)$ still has $\alpha$ as a simple root. And since it is trivially an even polynomial, it satisfies the criteria.

Now suppose that $\mu_{\alpha}$ does have $-\alpha$ as a root. Then $\mu_{\alpha}$ is even as otherwise $\mu_{\alpha}(x)$ and $\mu_{\alpha}(-x)$ would be linearly independent polynomials of the same degree, both of which have $\alpha$ as a root. Thus, they would have a linear combination with a smaller degree, and which still has $\alpha$ as a root, which contradicts $\mu_{\alpha}$ being the minimal polynomial of $\alpha$. Thus, $\omega_{\alpha}(x)=\mu_{\alpha}(x)$ satisfies the conditions.

We can assume without loss of generality that $\omega_{\alpha}(x)$ has a positive leading coefficient, otherwise we take its additive inverse.
\end{proof}

Now continue the proof of Proposition \ref{prop:algebra}. Take a partition of $\sigma$ into a set $\mathcal{I}$ of disjoint intervals such that $\left\lvert\mathcal{I}\right\rvert$ is minimal and denote the finite positive numbers that occur as the endpoint of at least one interval from $\mathcal{I}$ by $\rho_1<\rho_2<...<\rho_n$. This means that using the notations $\rho_0=0$ and $\rho_{n+1}=+\infty$, the intervals $\left(\rho_i,\rho_{i+1}\right) (0\le i\le n)$ either fully belong to $\sigma$ or are fully outside of it, otherwise there would be an endpoint in the interior of the particular interval, which contradicts the assumption that $\rho_1,\rho_2,...,\rho_n$ are the only finite positive endpoints.

Now we will separate six types of $\rho_i$:

1) $\left(\rho_{i-1},\rho_i\right)\nsubseteq\sigma$, $\rho_i\notin\sigma$ and $\left(\rho_i,\rho_{i+1}\right)\subseteq\sigma$

2) $\left(\rho_{i-1},\rho_i\right)\nsubseteq\sigma$, $\rho_i\in\sigma$ and $\left(\rho_i,\rho_{i+1}\right)\nsubseteq\sigma$

3) $\left(\rho_{i-1},\rho_i\right)\nsubseteq\sigma$, $\rho_i\in\sigma$ and $\left(\rho_i,\rho_{i+1}\right)\subseteq\sigma$

4) $\left(\rho_{i-1},\rho_i\right)\subseteq\sigma$, $\rho_i\notin\sigma$ and $\left(\rho_i,\rho_{i+1}\right)\nsubseteq\sigma$

5) $\left(\rho_{i-1},\rho_i\right)\subseteq\sigma$, $\rho_i\notin\sigma$ and $\left(\rho_i,\rho_{i+1}\right)\subseteq\sigma$

6) $\left(\rho_{i-1},\rho_i\right)\subseteq\sigma$, $\rho_i\in\sigma$ and $\left(\rho_i,\rho_{i+1}\right)\nsubseteq\sigma$

The above are the only possibilities, since because of the minimality of $n$, it is not possible for any $\rho_i$ ($1\le i\le n$) that $\left(\rho_{i-1},\rho_i\right)\nsubseteq\sigma$, $\rho_i\nsubseteq\sigma$ and $\left(\rho_i,\rho_{i+1}\right)\nsubseteq\sigma$ and similarly, it is also not possible that $\left(\rho_{i-1},\rho_i\right)\subseteq\sigma$, $\rho_i\subseteq\sigma$ and $\left(\rho_i,\rho_{i+1}\right)\subseteq\sigma$.

Now we will construct polynomials $p_i$ ($1\le i\le n$) and numbers $\zeta_i$ using the following table:

\begin{center}
\begin{tabular}[t]{ | c | c | c | c | }
\hline
$\rho_i$&$p_i(x)$ if $\omega'_{\rho_i}(\rho_i)>0$&$p_i(x)$ if $\omega'_{\rho_i}(\rho_i)<0$&$\zeta_i$\\
\hline
Type 1&$-\omega_{\rho_i}(x)\cdot\left(x^2-\left(\left\lfloor\rho_i\right\rfloor+1\right)^2\right)$&$-\omega_{\rho_i}(x)$&$1$\\
\hline
Type 2&$-\left(\omega_{\rho_i}(x)\right)^2$&$-\left(\omega_{\rho_i}(x)\right)^2$&$0$\\
\hline
Type 3&$-\omega_{\rho_i}(x)\cdot\left(x^2-\left(\left\lfloor\rho_i\right\rfloor+1\right)^2\right)$&$-\omega_{\rho_i}(x)$&$0$\\
\hline
Type 4&$-\omega_{\rho_i}(x)$&$-\omega_{\rho_i}(x)\cdot\left(x^2-\left(\left\lfloor\rho_i\right\rfloor+1\right)^2\right)$&$1$\\
\hline
Type 5&$-\left(\omega_{\rho_i}(x)\right)^2\cdot\left(x^2-\left(\left\lfloor\rho_i\right\rfloor+1\right)^2\right)$&$-\left(\omega_{\rho_i}(x)\right)^2\cdot\left(x^2-\left(\left\lfloor\rho_i\right\rfloor+1\right)^2\right)$&$1$\\
\hline
Type 6&$-\omega_{\rho_i}(x)$&$-\omega_{\rho_i}(x)\cdot\left(x^2-\left(\left\lfloor\rho_i\right\rfloor+1\right)^2\right)$&$0$\\
\hline
\end{tabular}
\end{center}

(Note that since $\rho_i$ is a simple root of $\omega'_{\rho_i}$, $\omega'_{\rho_i}\left(\rho_i\right)$ cannot be $0$.)

These polynomials are even and have a negative leading coefficient.

Also, for all $i$ ($1\le i\le n$), $\omega_{\rho_i}$ has $\rho_i$ as a simple root, while $\left(\omega_{\rho_i}\right)^2$ has $\rho_i$ as a root with multiplicity $2$. And $\left(x^2-\left(\left\lfloor\rho_i\right\rfloor+1\right)^2\right)$ has exactly two roots: $\pm\left(\left\lfloor\rho_i\right\rfloor+1\right)$. Thus, multiplying a polynomial with $\left(x^2-\left(\left\lfloor\rho_i\right\rfloor+1\right)^2\right)$ does not change the multiplicity of $\rho_i$ as its root, but reverses the sign of its derivative in $\rho_i$. Also, taking the additive inverse of a polynomial also does not change the multiplicity of $\rho_i$ as its root, but reverses the sign of its derivative in $\rho_i$.

And these properties combined prove that for a sufficiently small neighbourhood $N_i$ of $\rho_i$, $S_{\zeta_i}\left(p,0,+\infty\right)\cap N_i=\sigma\cap N_i$. Thus, if we choose $L_i$ and $U_i$ as arbitrary rational numbers within $\left(\rho_{i-1},\rho_i\right)\cap N_i$ and $\left(\rho_i,\rho_{i+1}\right)\cap N_i$, respectively, then $S_{\zeta_i}\left(p_i,L_i,U_i\right)\cap\left(L_i,U_i\right)=\sigma\cap\left(L_i,U_i\right)$ and $S_{\zeta_i}\left(p_i,L_i,U_i\right)$ covers all numbers outside $\left(L_i,U_i\right)$. Thus, the only possible differences of $\bigcap\limits_{i=1}^n{S_{\zeta_i}\left(p_i,L_i,U_i\right)}$ from $\sigma$ are those $\left(L_i,U_{i+1}\right)$ intervals, which exist ($U_{i+1}$ can be smaller than $L_i$) and are not part of $\sigma$. So if we take $p_{n+1}(x)$ as a multiple of the rational polynomial $\left(x^2-L_1^2\right)\cdot\left(x^2-U_n^2\right)\cdot\prod\limits_{i=1,\left(\rho_i,\rho_{i+1}\right)\in\sigma}^{n-1}{\left(x^2-U_i^2\right)\cdot\left(x^2-L_{i+1}^2\right)}$ so that $p_{n+1}(x)$ has integer coefficients and a negative leading coefficient and we also take $\zeta_{n+1}=0$, then the statement of Proposition \ref{prop:algebra} is satisfied.
\end{proof}

\section{Proof of Theorem \ref{thm:main}}\label{sec:proof}

\begin{proposition}\label{prop:main}
For any even polynomial $p\in\mathbb{Z}[x]$ with integer coefficients and a negative leading coefficient, there exists an EBG $G\left(p\right)$, whose range is $S_0\left(p,0,+\infty\right)$.\hfill\hyperlink{pf:main}{Proof}
\end{proposition}

{\it Sketch of the proof:} We define partly virtual EBGs (PVEBG), in which we also allow directed green edges, divided into groups. In a $(1,d)$-representation of a PVEBG, we require green edges from the same group to have the same vector, besides the criteria for EBGs and we define the range of PVEBGs analogously to EBGs. In case some boundedness conditions apply, an EBG with the same range can be created by connecting green edges by red grids.

The most crucial component of creating $G(p)$ is graph $A$ (Figure \ref{rrrbcycles}).

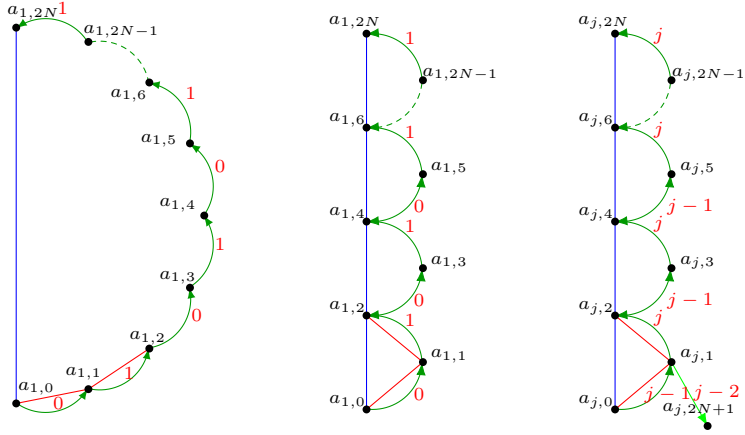
\begin{figure}
\begin{subfigure}
	\centering
	\usetikzlibrary{arrows}
\definecolor{qqzzqq}{rgb}{0,0.6,0}
\definecolor{qqqqff}{rgb}{0,0,1}
\definecolor{ffqqqq}{rgb}{1,0,0}
\begin{tikzpicture}[line cap=round,line join=round]
\clip(-0.36,-0.48) rectangle (3.22,5.86);
\draw [color=ffqqqq] (0,0)-- (0.96,0.19);
\draw [color=ffqqqq] (0.96,0.19)-- (1.77,0.73);
\draw [color=qqqqff] (0,5)-- (0,0);
\draw [shift={(0.39,0.54)},color=qqzzqq,-latex]  plot[domain=4.09:5.73,variable=\t]({1*0.66*cos(\t r)+0*0.66*sin(\t r)},{0*0.66*cos(\t r)+1*0.66*sin(\t r)});
\draw [shift={(1.11,0.84)},color=qqzzqq,-latex]  plot[domain=4.48:6.12,variable=\t]({1*0.66*cos(\t r)+0*0.66*sin(\t r)},{0*0.66*cos(\t r)+1*0.66*sin(\t r)});
\draw [shift={(1.66,1.39)},color=qqzzqq,-latex]  plot[domain=-1.41:0.23,variable=\t]({1*0.66*cos(\t r)+0*0.66*sin(\t r)},{0*0.66*cos(\t r)+1*0.66*sin(\t r)});
\draw [shift={(1.96,2.11)},color=qqzzqq,-latex]  plot[domain=-1.02:0.63,variable=\t]({1*0.66*cos(\t r)+0*0.66*sin(\t r)},{0*0.66*cos(\t r)+1*0.66*sin(\t r)});
\draw [shift={(1.96,2.89)},color=qqzzqq,-latex]  plot[domain=-0.63:1.02,variable=\t]({1*0.66*cos(\t r)+0*0.66*sin(\t r)},{0*0.66*cos(\t r)+1*0.66*sin(\t r)});
\draw [shift={(1.66,3.61)},color=qqzzqq,-latex]  plot[domain=-0.23:1.41,variable=\t]({1*0.66*cos(\t r)+0*0.66*sin(\t r)},{0*0.66*cos(\t r)+1*0.66*sin(\t r)});
\draw [shift={(1.11,4.16)},dash pattern=on 2pt off 2pt,color=qqzzqq]  plot[domain=0.16:1.81,variable=\t]({1*0.66*cos(\t r)+0*0.66*sin(\t r)},{0*0.66*cos(\t r)+1*0.66*sin(\t r)});
\draw [shift={(0.39,4.46)},color=qqzzqq,-latex]  plot[domain=0.55:2.2,variable=\t]({1*0.66*cos(\t r)+0*0.66*sin(\t r)},{0*0.66*cos(\t r)+1*0.66*sin(\t r)});
\begin{scriptsize}
\fill [color=black] (0,0) circle (1.5pt);
\draw[color=black] (0.27,0.2) node {$a_{1,0}$};
\fill [color=black] (0,5) circle (1.5pt);
\draw[color=black] (0.23,5.2) node {$a_{1,2N}$};
\fill [color=black] (0.96,0.19) circle (1.5pt);
\draw[color=black] (0.9,0.4) node {$a_{1,1}$};
\fill [color=black] (1.77,0.73) circle (1.5pt);
\draw[color=black] (1.75,0.87) node {$a_{1,2}$};
\fill [color=black] (2.31,1.54) circle (1.5pt);
\draw[color=black] (2.15,1.68) node {$a_{1,3}$};
\fill [color=black] (2.5,2.5) circle (1.5pt);
\draw[color=black] (2.15,2.64) node {$a_{1,4}$};
\fill [color=black] (2.31,3.46) circle (1.5pt);
\draw[color=black] (1.9,3.5) node {$a_{1,5}$};
\fill [color=black] (1.77,4.27) circle (1.5pt);
\draw[color=black] (1.5,4.1) node {$a_{1,6}$};
\fill [color=black] (0.96,4.81) circle (1.5pt);
\draw[color=black] (1.4,4.95) node {$a_{1,2N-1}$};
\draw[color=red] (0.56,0) node {0};
\draw[color=red] (1.5,0.4) node {1};
\draw[color=red] (2.4,1.17) node {0};
\draw[color=red] (2.71,2.12) node {1};
\draw[color=red] (2.71,3.16) node {0};
\draw[color=red] (2.32,4.13) node {1};
\draw[color=red] (0.62,5.26) node {1};
\end{scriptsize}
\end{tikzpicture}
\end{subfigure}
\begin{subfigure}
	\centering
	\usetikzlibrary{arrows}
\definecolor{qqzzqq}{rgb}{0,0.6,0}
\definecolor{ffqqqq}{rgb}{1,0,0}
\definecolor{uququq}{rgb}{0.25,0.25,0.25}
\definecolor{qqqqff}{rgb}{0,0,1}
\begin{tikzpicture}[line cap=round,line join=round]
\clip(-0.5,-0.4) rectangle (1.75,5.86);
\draw [color=qqqqff] (0,5)-- (0,0);
\draw [color=ffqqqq] (0,0)-- (0.75,0.63);
\draw [color=ffqqqq] (0.75,0.63)-- (0,1.25);
\draw [shift={(0.08,0.66)},color=qqzzqq,-Latex]  plot[domain=4.58:6.23,variable=\t]({1*0.66*cos(\t r)+0*0.66*sin(\t r)},{0*0.66*cos(\t r)+1*0.66*sin(\t r)});
\draw [shift={(0.08,1.91)},color=qqzzqq,-Latex]  plot[domain=4.58:6.23,variable=\t]({1*0.66*cos(\t r)+0*0.66*sin(\t r)},{0*0.66*cos(\t r)+1*0.66*sin(\t r)});
\draw [shift={(0.08,3.16)},color=qqzzqq,-Latex]  plot[domain=4.58:6.23,variable=\t]({1*0.66*cos(\t r)+0*0.66*sin(\t r)},{0*0.66*cos(\t r)+1*0.66*sin(\t r)});
\draw [shift={(0.08,4.41)},dash pattern=on 2pt off 2pt,color=qqzzqq]  plot[domain=4.58:6.23,variable=\t]({1*0.66*cos(\t r)+0*0.66*sin(\t r)},{0*0.66*cos(\t r)+1*0.66*sin(\t r)});
\draw [shift={(0.08,0.59)},color=qqzzqq,-Latex]  plot[domain=0.05:1.7,variable=\t]({1*0.66*cos(\t r)+0*0.66*sin(\t r)},{0*0.66*cos(\t r)+1*0.66*sin(\t r)});
\draw [shift={(0.08,1.84)},color=qqzzqq,-Latex]  plot[domain=0.05:1.7,variable=\t]({1*0.66*cos(\t r)+0*0.66*sin(\t r)},{0*0.66*cos(\t r)+1*0.66*sin(\t r)});
\draw [shift={(0.08,3.09)},color=qqzzqq,-Latex]  plot[domain=0.05:1.7,variable=\t]({1*0.66*cos(\t r)+0*0.66*sin(\t r)},{0*0.66*cos(\t r)+1*0.66*sin(\t r)});
\draw [shift={(0.08,4.34)},color=qqzzqq,-Latex]  plot[domain=0.05:1.7,variable=\t]({1*0.66*cos(\t r)+0*0.66*sin(\t r)},{0*0.66*cos(\t r)+1*0.66*sin(\t r)});
\begin{scriptsize}
\fill [color=black] (0,0) circle (1.5pt);
\draw[color=black] (-0.25,0.1) node {$a_{1,0}$};
\fill [color=black] (0,5) circle (1.5pt);
\draw[color=black] (-0.15,5.15) node {$a_{1,2N}$};
\fill [color=black] (0.75,0.63) circle (1.5pt);
\draw[color=black] (1.1,0.71) node {$a_{1,1}$};
\fill [color=black] (0,1.25) circle (1.5pt);
\draw[color=black] (-0.25,1.35) node {$a_{1,2}$};
\fill [color=black] (0.75,1.88) circle (1.5pt);
\draw[color=black] (1.1,1.96) node {$a_{1,3}$};
\fill [color=black] (0,2.5) circle (1.5pt);
\draw[color=black] (-0.25,2.6) node {$a_{1,4}$};
\fill [color=black] (0.75,3.13) circle (1.5pt);
\draw[color=black] (1.1,3.21) node {$a_{1,5}$};
\fill [color=black] (0,3.75) circle (1.5pt);
\draw[color=black] (-0.25,3.85) node {$a_{1,6}$};
\fill [color=black] (0.75,4.38) circle (1.5pt);
\draw[color=black] (1.25,4.46) node {$a_{1,2N-1}$};
\draw[color=red] (0.7,0.2) node {0};
\draw[color=red] (0.7,1.45) node {0};
\draw[color=red] (0.7,2.7) node {0};
\draw[color=red] (0.58,1.19) node {1};
\draw[color=red] (0.58,2.44) node {1};
\draw[color=red] (0.58,3.69) node {1};
\draw[color=red] (0.58,4.94) node {1};
\end{scriptsize}
\end{tikzpicture}
\end{subfigure}
\begin{subfigure}
	\centering
	\usetikzlibrary{arrows}
\definecolor{qqffqq}{rgb}{0,1,0}
\definecolor{qqzzqq}{rgb}{0,0.6,0}
\definecolor{ffqqqq}{rgb}{1,0,0}
\definecolor{uququq}{rgb}{0.25,0.25,0.25}
\definecolor{qqqqff}{rgb}{0,0,1}
\begin{tikzpicture}[line cap=round,line join=round]
\clip(-0.5,-0.4) rectangle (1.75,5.86);
\draw [color=qqqqff] (0,5)-- (0,0);
\draw [color=ffqqqq] (0,0)-- (0.75,0.63);
\draw [color=ffqqqq] (0.75,0.63)-- (0,1.25);
\draw [shift={(0.08,0.66)},color=qqzzqq,-Latex]  plot[domain=4.58:6.23,variable=\t]({1*0.66*cos(\t r)+0*0.66*sin(\t r)},{0*0.66*cos(\t r)+1*0.66*sin(\t r)});
\draw [shift={(0.08,1.91)},color=qqzzqq,-Latex]  plot[domain=4.58:6.23,variable=\t]({1*0.66*cos(\t r)+0*0.66*sin(\t r)},{0*0.66*cos(\t r)+1*0.66*sin(\t r)});
\draw [shift={(0.08,3.16)},color=qqzzqq,-Latex]  plot[domain=4.58:6.23,variable=\t]({1*0.66*cos(\t r)+0*0.66*sin(\t r)},{0*0.66*cos(\t r)+1*0.66*sin(\t r)});
\draw [shift={(0.08,4.41)},dash pattern=on 2pt off 2pt,color=qqzzqq]  plot[domain=4.58:6.23,variable=\t]({1*0.66*cos(\t r)+0*0.66*sin(\t r)},{0*0.66*cos(\t r)+1*0.66*sin(\t r)});
\draw [shift={(0.08,0.59)},color=qqzzqq,-Latex]  plot[domain=0.05:1.7,variable=\t]({1*0.66*cos(\t r)+0*0.66*sin(\t r)},{0*0.66*cos(\t r)+1*0.66*sin(\t r)});
\draw [shift={(0.08,1.84)},color=qqzzqq,-Latex]  plot[domain=0.05:1.7,variable=\t]({1*0.66*cos(\t r)+0*0.66*sin(\t r)},{0*0.66*cos(\t r)+1*0.66*sin(\t r)});
\draw [shift={(0.08,3.09)},color=qqzzqq,-Latex]  plot[domain=0.05:1.7,variable=\t]({1*0.66*cos(\t r)+0*0.66*sin(\t r)},{0*0.66*cos(\t r)+1*0.66*sin(\t r)});
\draw [shift={(0.08,4.34)},color=qqzzqq,-Latex]  plot[domain=0.05:1.7,variable=\t]({1*0.66*cos(\t r)+0*0.66*sin(\t r)},{0*0.66*cos(\t r)+1*0.66*sin(\t r)});
\draw [color=qqffqq,Latex-] (1.23,-0.22)-- (0.75,0.63);
\begin{scriptsize}
\fill [color=black] (0,0) circle (1.5pt);
\draw[color=black] (-0.25,0.1) node {$a_{j,0}$};
\fill [color=black] (0,5) circle (1.5pt);
\draw[color=black] (-0.15,5.15) node {$a_{j,2N}$};
\fill [color=black] (0.75,0.63) circle (1.5pt);
\draw[color=black] (1.1,0.71) node {$a_{j,1}$};
\fill [color=black] (0,1.25) circle (1.5pt);
\draw[color=black] (-0.25,1.35) node {$a_{j,2}$};
\fill [color=black] (0.75,1.88) circle (1.5pt);
\draw[color=black] (1.1,1.96) node {$a_{j,3}$};
\fill [color=black] (0,2.5) circle (1.5pt);
\draw[color=black] (-0.25,2.6) node {$a_{j,4}$};
\fill [color=black] (0.75,3.13) circle (1.5pt);
\draw[color=black] (1.1,3.21) node {$a_{j,5}$};
\fill [color=black] (0,3.75) circle (1.5pt);
\draw[color=black] (-0.25,3.85) node {$a_{j,6}$};
\fill [color=black] (0.75,4.38) circle (1.5pt);
\draw[color=black] (1.25,4.46) node {$a_{j,2N-1}$};
\draw[color=ffqqqq] (0.7,0.2) node {$j-1$};
\draw[color=ffqqqq] (1,1.45) node {$j-1$};
\draw[color=ffqqqq] (1,2.7) node {$j-1$};
\draw[color=ffqqqq] (0.58,1.22) node {$j$};
\draw[color=ffqqqq] (0.58,2.47) node {$j$};
\draw[color=ffqqqq] (0.58,3.72) node {$j$};
\draw[color=ffqqqq] (0.58,4.97) node {$j$};
\draw[color=ffqqqq] (1.35,0.22) node {$j-2$};
\fill [color=black] (1.23,-0.22) circle (1.5pt);
\draw[color=black] (1.12,0) node {$a_{j,2N+1}$};
\end{scriptsize}
\end{tikzpicture}
\end{subfigure}
\caption{The components of $A$: $A_1$ (left), the only $(1,d)$-representations (up to isometry) of $A_1$ (middle) and of $A_j$ ($2\le j\le deg(p)$) (right) ($N$ is large enough and groups are denoted by numbers).}
\label{rrrbcycles}
\end{figure}

For small enough $d$, $A$ has exactly one $(1,d)$-representation up to transformations which are isometries on the components. If we draw the complex plane so that $\vec{a_{1,0}a_{1,1}}=1$ and $\vec{a_{1,1}a_{1,2}}=\varepsilon$, the members of the group marked by $j$ will have vector $\varepsilon^j$ and $\left\lvert N\cdot\left(1+\varepsilon\right)\right\rvert=d$. This helps constructing points having distance of some even polynomial of $d$, and ultimatley, constructing $G(p)$. 

\begin{proposition}\label{prop:s1}
For any even polynomial $p\in\mathbb{Z}[x]$ with a negative leading coefficient, there exists an EBG $G'\left(p\right)$, whose range is $S_1(p,0,+\infty)$.\hfill\hyperlink{pf:s1}{Proof}
\end{proposition}

\begin{proposition}\label{prop:valami2}
For an EBG $G$, positive rational numbers $L_a, U_a$ and arbitrary real numbers $L_b, U_b$ ($L_b<L_a<U_a<U_b$), if $ran(G)\cap\left(L_a,U_a\right]\neq\emptyset$, then there exists an EBG $G_{L_a,L_b}^{U_a,U_b}$ for which $ran\left(G_{L_a,L_b}^{U_a,U_b}\right)\cap\left(L_b,U_b\right)=\left(\left(0,L_a\right]\cup ran(G)\cup\left[U_a,+\infty\right)\right)\cap\left(L_b,U_b\right)$.\hfill\hyperlink{pf:valami2}{Proof}
\end{proposition}

Using the notations of Proposition \ref{prop:algebra}, with the help of Proposition \ref{prop:valami2}, we can construct $(1,d)$-graphs $G\left(p_i\right)_{\lambda,L_i}^{\upsilon,U_i}$ for $1\le i\le n$ and $\zeta_i=0$, while in case of $\zeta_i=1$, we construct $G'\left(p_i\right)_{\lambda,L_i}^{\upsilon,U_i}$, whose range coincides with $S_{\zeta_i}(p_i,L_i,U_i)$ on the interval $\left[\lambda,\upsilon\right]$. And finally, we can take $G'\left(p_{n+1}\right)$, whose range is empty outside of $\left[\lambda,\upsilon\right]$, thus the intersection of the ranges of these graphs is $\sigma$ because of Proposition \ref{prop:algebra}. Thus, because of Lemma \ref{lem:union}, their disjoint union has $\sigma$ as its range. So all semialgebraic sets from $\left[\lambda,\upsilon\right]$ are the range of some EBG.\hfill\qed

\subsection{Introducing partly virtual edge-bicoloured graphs}\label{subsec:pvebg}

This section mostly contains long proofs for some technical statements.

\begin{center}
\centering
\usetikzlibrary{arrows}
\definecolor{ffqqqq}{rgb}{1,0,0}
\begin{tikzpicture}[line cap=round,line join=round]
\clip(-0.94,-0.64) rectangle (8.32,6.28);
\draw [color=ffqqqq] (0,0)-- (0.77,0.64);
\draw [color=ffqqqq] (0.77,0.64)-- (1.67,0.21);
\draw [color=ffqqqq] (1.67,0.21)-- (2.67,0.24);
\draw [color=ffqqqq] (2.67,0.24)-- (3.61,0.58);
\draw [color=ffqqqq] (3.61,0.58)-- (4.47,0.07);
\draw [color=ffqqqq] (4.47,0.07)-- (5.47,0.15);
\draw [color=ffqqqq] (5.47,0.15)-- (6.41,-0.19);
\draw [color=ffqqqq] (6.41,-0.19)-- (7.31,0.25);
\draw [color=ffqqqq] (7.31,0.25)-- (6.95,1.18);
\draw [color=ffqqqq] (6.95,1.18)-- (7.36,2.09);
\draw [color=ffqqqq] (7.36,2.09)-- (7.12,3.06);
\draw [color=ffqqqq] (7.12,3.06)-- (7.05,4.06);
\draw [color=ffqqqq] (7.05,4.06)-- (7.33,5.02);
\draw [color=ffqqqq] (7.33,5.02)-- (6.43,4.59);
\draw [color=ffqqqq] (6.43,4.59)-- (5.49,4.92);
\draw [color=ffqqqq] (5.49,4.92)-- (4.5,4.85);
\draw [color=ffqqqq] (4.5,4.85)-- (3.63,5.36);
\draw [color=ffqqqq] (3.63,5.36)-- (2.69,5.02);
\draw [color=ffqqqq] (2.69,5.02)-- (1.69,4.99);
\draw [color=ffqqqq] (1.69,4.99)-- (0.79,5.42);
\draw [color=ffqqqq] (0.79,5.42)-- (0.02,4.78);
\draw [color=ffqqqq] (0.02,4.78)-- (-0.26,3.82);
\draw [color=ffqqqq] (-0.26,3.82)-- (-0.19,2.82);
\draw [color=ffqqqq] (-0.19,2.82)-- (0.05,1.85);
\draw [color=ffqqqq] (0.05,1.85)-- (-0.36,0.93);
\draw [color=ffqqqq] (-0.36,0.93)-- (0,0);
\draw [color=ffqqqq] (0.77,0.64)-- (0.41,1.57);
\draw [color=ffqqqq] (0.41,1.57)-- (0.81,2.49);
\draw [color=ffqqqq] (0.81,2.49)-- (0.58,3.46);
\draw [color=ffqqqq] (0.58,3.46)-- (0.51,4.46);
\draw [color=ffqqqq] (0.51,4.46)-- (0.79,5.42);
\draw [color=ffqqqq] (1.69,4.99)-- (1.41,4.03);
\draw [color=ffqqqq] (1.41,4.03)-- (1.48,3.03);
\draw [color=ffqqqq] (1.48,3.03)-- (1.72,2.06);
\draw [color=ffqqqq] (1.72,2.06)-- (1.31,1.14);
\draw [color=ffqqqq] (1.31,1.14)-- (1.67,0.21);
\draw [color=ffqqqq] (2.67,0.24)-- (2.31,1.17);
\draw [color=ffqqqq] (2.31,1.17)-- (2.72,2.09);
\draw [color=ffqqqq] (2.72,2.09)-- (2.48,3.06);
\draw [color=ffqqqq] (2.48,3.06)-- (2.41,4.06);
\draw [color=ffqqqq] (2.41,4.06)-- (2.69,5.02);
\draw [color=ffqqqq] (3.63,5.36)-- (3.35,4.4);
\draw [color=ffqqqq] (3.35,4.4)-- (3.42,3.4);
\draw [color=ffqqqq] (3.42,3.4)-- (3.66,2.43);
\draw [color=ffqqqq] (3.66,2.43)-- (3.25,1.51);
\draw [color=ffqqqq] (3.25,1.51)-- (3.61,0.58);
\draw [color=ffqqqq] (4.47,0.07)-- (4.11,1.01);
\draw [color=ffqqqq] (4.11,1.01)-- (4.52,1.92);
\draw [color=ffqqqq] (4.52,1.92)-- (4.29,2.89);
\draw [color=ffqqqq] (4.29,2.89)-- (4.21,3.89);
\draw [color=ffqqqq] (4.21,3.89)-- (4.5,4.85);
\draw [color=ffqqqq] (5.49,4.92)-- (5.21,3.97);
\draw [color=ffqqqq] (5.21,3.97)-- (5.28,2.97);
\draw [color=ffqqqq] (5.28,2.97)-- (5.52,2);
\draw [color=ffqqqq] (5.52,2)-- (5.11,1.08);
\draw [color=ffqqqq] (5.11,1.08)-- (5.47,0.15);
\draw [color=ffqqqq] (6.41,-0.19)-- (6.05,0.74);
\draw [color=ffqqqq] (6.05,0.74)-- (6.46,1.66);
\draw [color=ffqqqq] (6.46,1.66)-- (6.22,2.63);
\draw [color=ffqqqq] (6.22,2.63)-- (6.15,3.63);
\draw [color=ffqqqq] (6.15,3.63)-- (6.43,4.59);
\draw [color=ffqqqq] (7.05,4.06)-- (6.15,3.63);
\draw [color=ffqqqq] (6.15,3.63)-- (5.21,3.97);
\draw [color=ffqqqq] (5.21,3.97)-- (4.21,3.89);
\draw [color=ffqqqq] (4.21,3.89)-- (3.35,4.4);
\draw [color=ffqqqq] (3.35,4.4)-- (2.41,4.06);
\draw [color=ffqqqq] (2.41,4.06)-- (1.41,4.03);
\draw [color=ffqqqq] (1.41,4.03)-- (0.51,4.46);
\draw [color=ffqqqq] (0.51,4.46)-- (-0.26,3.82);
\draw [color=ffqqqq] (-0.19,2.82)-- (0.58,3.46);
\draw [color=ffqqqq] (0.58,3.46)-- (1.48,3.03);
\draw [color=ffqqqq] (1.48,3.03)-- (2.48,3.06);
\draw [color=ffqqqq] (2.48,3.06)-- (3.42,3.4);
\draw [color=ffqqqq] (3.42,3.4)-- (4.29,2.89);
\draw [color=ffqqqq] (4.29,2.89)-- (5.28,2.97);
\draw [color=ffqqqq] (5.28,2.97)-- (6.22,2.63);
\draw [color=ffqqqq] (6.22,2.63)-- (7.12,3.06);
\draw [color=ffqqqq] (7.36,2.09)-- (6.46,1.66);
\draw [color=ffqqqq] (6.46,1.66)-- (5.52,2);
\draw [color=ffqqqq] (5.52,2)-- (4.52,1.92);
\draw [color=ffqqqq] (4.52,1.92)-- (3.66,2.43);
\draw [color=ffqqqq] (3.66,2.43)-- (2.72,2.09);
\draw [color=ffqqqq] (2.72,2.09)-- (1.72,2.06);
\draw [color=ffqqqq] (1.72,2.06)-- (0.81,2.49);
\draw [color=ffqqqq] (0.81,2.49)-- (0.05,1.85);
\draw [color=ffqqqq] (-0.36,0.93)-- (0.41,1.57);
\draw [color=ffqqqq] (0.41,1.57)-- (1.31,1.14);
\draw [color=ffqqqq] (1.31,1.14)-- (2.31,1.17);
\draw [color=ffqqqq] (2.31,1.17)-- (3.25,1.51);
\draw [color=ffqqqq] (3.25,1.51)-- (4.11,1.01);
\draw [color=ffqqqq] (4.11,1.01)-- (5.11,1.08);
\draw [color=ffqqqq] (5.11,1.08)-- (6.05,0.74);
\draw [color=ffqqqq] (6.05,0.74)-- (6.95,1.18);
\begin{scriptsize}
\fill [color=black] (0,0) circle (1.5pt);
\draw[color=black] (0.5,0) node {$u_{n,0}$};
\fill [color=black] (0.77,0.64) circle (1.5pt);
\draw[color=black] (1.15,0.75) node {$u_{n,1}$};
\fill [color=black] (1.67,0.21) circle (1.5pt);
\draw[color=black] (2.1,0.48) node {$u_{n,2}$};
\fill [color=black] (2.67,0.24) circle (1.5pt);
\fill [color=black] (3.61,0.58) circle (1.5pt);
\fill [color=black] (4.47,0.07) circle (1.5pt);
\fill [color=black] (5.47,0.15) circle (1.5pt);
\fill [color=black] (6.41,-0.19) circle (1.5pt);
\fill [color=black] (7.31,0.25) circle (1.5pt);
\draw[color=black] (8.04,0.5) node {$u_{n,k}$};
\fill [color=black] (-0.36,0.93) circle (1.5pt);
\fill [color=black] (0.05,1.85) circle (1.5pt);
\fill [color=black] (-0.19,2.82) circle (1.5pt);
\draw[color=black] (0.2,2.75) node {$u_{2,0}$};
\fill [color=black] (-0.26,3.82) circle (1.5pt);
\draw[color=black] (0.15,3.7) node {$u_{1,0}$};
\fill [color=black] (0.02,4.78) circle (1.5pt);
\draw[color=black] (0,5.04) node {$u_{0,0}$};
\fill [color=black] (0.41,1.57) circle (1.5pt);
\fill [color=black] (1.31,1.14) circle (1.5pt);
\fill [color=black] (2.31,1.17) circle (1.5pt);
\fill [color=black] (3.25,1.51) circle (1.5pt);
\fill [color=black] (4.11,1.01) circle (1.5pt);
\fill [color=black] (5.11,1.08) circle (1.5pt);
\fill [color=black] (6.05,0.74) circle (1.5pt);
\fill [color=black] (6.95,1.18) circle (1.5pt);
\fill [color=black] (0.81,2.49) circle (1.5pt);
\fill [color=black] (1.72,2.06) circle (1.5pt);
\fill [color=black] (2.72,2.09) circle (1.5pt);
\fill [color=black] (3.66,2.43) circle (1.5pt);
\fill [color=black] (4.52,1.92) circle (1.5pt);
\fill [color=black] (5.52,2) circle (1.5pt);
\fill [color=black] (6.46,1.66) circle (1.5pt);
\fill [color=black] (7.36,2.09) circle (1.5pt);
\fill [color=black] (0.58,3.46) circle (1.5pt);
\draw[color=black] (1,3.5) node {$u_{2,1}$};
\fill [color=black] (1.48,3.03) circle (1.5pt);
\draw[color=black] (1.9,3.28) node {$u_{2,2}$};
\fill [color=black] (2.48,3.06) circle (1.5pt);
\fill [color=black] (3.42,3.4) circle (1.5pt);
\fill [color=black] (4.29,2.89) circle (1.5pt);
\fill [color=black] (5.28,2.97) circle (1.5pt);
\fill [color=black] (6.22,2.63) circle (1.5pt);
\fill [color=black] (7.12,3.06) circle (1.5pt);
\draw[color=black] (7.84,3.32) node {$u_{2,k}$};
\fill [color=black] (0.51,4.46) circle (1.5pt);
\draw[color=black] (1,4.5) node {$u_{1,1}$};
\fill [color=black] (1.41,4.03) circle (1.5pt);
\draw[color=black] (1.8,4.28) node {$u_{1,2}$};
\fill [color=black] (2.41,4.06) circle (1.5pt);
\fill [color=black] (3.35,4.4) circle (1.5pt);
\fill [color=black] (4.21,3.89) circle (1.5pt);
\fill [color=black] (5.21,3.97) circle (1.5pt);
\fill [color=black] (6.15,3.63) circle (1.5pt);
\fill [color=black] (7.05,4.06) circle (1.5pt);
\draw[color=black] (7.78,4.32) node {$u_{1,k}$};
\fill [color=black] (0.79,5.42) circle (1.5pt);
\draw[color=black] (1.25,5.4) node {$u_{0,1}$};
\fill [color=black] (1.69,4.99) circle (1.5pt);
\draw[color=black] (1.75,5.24) node {$u_{0,2}$};
\fill [color=black] (2.69,5.02) circle (1.5pt);
\fill [color=black] (3.63,5.36) circle (1.5pt);
\fill [color=black] (4.5,4.85) circle (1.5pt);
\fill [color=black] (5.49,4.92) circle (1.5pt);
\fill [color=black] (6.43,4.59) circle (1.5pt);
\fill [color=black] (7.33,5.02) circle (1.5pt);
\draw[color=black] (8.06,5.28) node {$u_{0,k}$};
\end{scriptsize}
\end{tikzpicture}
\captionof{figure}{}
\label{rhombusgrid6}
\end{center}

\begin{definition}
In an EBG, let a red $n\times k$ grid be a subgraph determined by vertices $\left\lbrace u_{i,j}\vert0\le i\le n,0\le j\le k\right\rbrace$ in which the set of edges is
$$\left\lbrace\left(u_{i,j},u_{i,j+1}\right)\vert0\le i\le n,0\le j\le k-1\right\rbrace\cup\left\lbrace\left(u_{i,j}u_{i+1,j}\right)\vert0\le i\le n-1,0\le j\le k\right\rbrace$$
and all of them are coloured with red as it can be seen in Figure \ref{rhombusgrid6}. An isomorphic subgraph with its edges coloured with blue is called a blue $n\times k$ grid. Call the vertices $u_{0,0}$, $u_{n,0}$, $u_{n,k}$ and $u_{0,k}$ the corners of the grid.
\end{definition}

\begin{definition}
Define a partly virtual edge-bicoloured graph (PVEBG) in the following way:

It has three types of edges: red edges, blue edges and directed green edges. The green edges are assigned to equivalence classes, an equivalence class also has a colour of its own, which can be red or blue. A bijection from the vertices of a PVEBG to the plane is a $(1,d)$-representation of the PVEBG if and only if all vertices go to distinct points, all red edges correspond to segments of length $1$, all blue edges correspond to segments of length $d$ and all green edges from the same equivalence class correspond to directed segments of the same vector.

In the figures, we will denote red edges by straight red segments, blue edges by straight blue segments and green edges by straight green segments or green circular arcs. Green edges also get numbers, where two green edges being in the same class are denoted by numbers of the colour of the respective equivalence class.
\end{definition}

\begin{definition}
Define the range $ran(G)$ of a PVEBG $G$ as the set of numbers $d$ for which it has a $(1,d)$-representation and its $W$-range, denoted by $ran_W(G)$ for some $W\in\mathbb{Z}_{>0}$ as the set of numbers $d$ for which it has a $(1,d)$-representation in which all of its connected components have diameter less than $\min\left(W,Wd\right)$, all green edges belonging to a red equivalence class have length less than $W$ and all green edges belonging to a blue equivalence class have length less than $Wd$.
\end{definition}

Now we prove a long technical claim:

\begin{proposition}\label{prop:pvebgtoebg}
Suppose that for some finite $W$, a PVEBG $G$ has $ran_W(G)=ran(G)$. Now let $G'$ be the EBG obtained from $G$ by connecting all green edges in the same red equivalence class by a red $W\times W$ grid and all blue edges in the same blue equivalence class by a blue $W\times W$ grid (and deleting the green edges). Then $ran(G')=ran(G)$.
\end{proposition}

\begin{proof}

\begin{lemma}\label{lem:grid0}
If an EBG contains a red or blue $n\times k$ grid, then for any $(1,d)$-representation of the graph, $\vec{u_{0,0}u_{0,k}}=\vec{u_{n,0}u_{n,k}}$ holds, using the notations in Figure \ref{rhombusgrid6}.\hfill\hyperlink{pf:grid0}{Proof}
\end{lemma}

\begin{proof}
First, for any $4$-cycle $u_{i,j}u_{i+1,j}u_{i+1,j+1},u_{i,j+1}$ in the grid, the four vertices form a rhombus, since the only other quadrilaterals with four sides of equal length have at least two coincident vertices, which is not possible in a $(1,d)$-representation per definition.

Now using the notations of Figure \ref{rhombusgrid6}, we get that for any $1\le i\le n-1$, $\vec{u_{i,0}u_{i+1,0}}=\vec{u_{i,1}u_{i+1,1}}=...=\vec{u_{i,k}u_{i+1,k}}$. And by summing up the equalitites we got above, we get $\vec{u_{0,0}u_{n,0}}=\sum\limits_{i=0}^{n-1}{\vec{u_{i,0}u_{i+1,0}}}=\sum\limits_{i=0}^{n-1}{\vec{u_{i,k}u_{i+1,k}}}=\vec{u_{0,k}u_{n,k}}$.
\end{proof}

\begin{lemma}\label{lem:grid1}
If for the $(1,d)$-representation of any red or blue $n\times k$ grid $\left(n,k\in\mathbb{N}\right)$, the vertices from $\left\lbrace u_{i,0}\vert0\le i\le n\right\rbrace\cup\left\lbrace u_{0,j}\vert0\le j\le k\right\rbrace$ are given, then the only possible representation is when all $u_{i,j}=u_{i,0}+\vec{u_{0,0}u_{0,j}}$, and it is a $\left(1,d\right)$-representation iff

1) vectors $\vec{u_{i,0}u_{i+1,0}}$ and $\vec{u_{0,j}u_{0,j+1}}$ all have length $1$ if the grid is red and they all have length $d$ if the grid is blue

2) no two points of the set $\left\lbrace u_{i,0}+\vec{u_{0,0}u_{0,j}}\vert0\le i\le n,0\le j\le k\right\rbrace$ coincide.
\end{lemma}

\begin{proof}
If we apply Lemma \ref{lem:grid0} to the subgrid induced by the vertices from

$\left\lbrace u_{l,m}\vert0\le l\le i,0\le m\le j\right\rbrace$, we get $\vec{u_{0,0}u_{0,j}}=\vec{u_{i,0}u_{i,j}}$, which implies that the only possible representation is the one defined above.

Also, in case of a red grid, condition 1 is trivially necessary for the edges to have length 1. And it is satisfactory too, since if it applies, then all $\left\lvert u_{i,j+1}-u_{i,j}\right\rvert=$

$\left\lvert\left(u_{i,0}+\vec{u_{0,0}u_{0,j+1}}\right)-\left(u_{i,0}+\vec{u_{0,0}u_{0,j}}\right)\right\rvert=\left\lvert u_{0,j+1}-u_{0,j}\right\rvert=1$ and all $\left\lvert u_{i+1,j}-u_{i,j}\right\rvert=\left\lvert\left(u_{i+1,0}+\vec{u_{0,0}u_{0,j}}\right)-\left(u_{i,0}+\vec{u_{0,0}u_{0,j}}\right)\right\rvert=\left\lvert u_{i+1,0}-u_{i,0}\right\rvert=1$. For blue grids the proof goes the same way.

From here, it follows from the definition of $(1,d)$-graphs that with the locations defined above, we get a $(1,d)$-representation exactly if conditions 1 and 2 apply.
\end{proof}

\begin{lemma}\label{lem:path}
Suppose that we have a red $n$-path $\pi$ ($n\ge3$) with vertices $u_0$, $u_1$, ... $u_n$ with the location of $u_0$ and $u_n$ fixed so that $0<\left\lvert u_0u_n\right\rvert<n$.

Then there are infinitely many $(1,d)$-representations of $\pi$ called $u_0^{(i)}$, $u_1^{(i)}$, ..., $u_n^{(i)}$ ($i\in J$ for some index set $J$) that satisfy three properties:

1) $u_i^{(j)}\neq u_i^{(k)}$ for $i\in\left\lbrace0,1,...,n\right\rbrace,j,k\in J$

2) $\left\lvert u_0^{(j)}u_i^{(j)}\right\rvert=i$ for $0\le i\le\left\lfloor\frac{n}{2}\right\rfloor,j\in J$

3) $\left\lvert u_n^{(j)}u_i^{(j)}\right\rvert=n-i$ for $\left\lfloor\frac{n}{2}\right\rfloor+1\le i\le n, j\in J$.
\end{lemma}

\begin{proof}
We will construct such representations in the following way:

Denote the circle around $u_0$ with radius $\left\lfloor\frac{n}{2}\right\rfloor$ by $\gamma_a$ and the circle around $u_n$ with radius $\left\lceil\frac{n}{2}\right\rceil-1$ by $\gamma_b$. First, we will search for a pair of points $P_a\in\gamma_a$ and $P_b\in\gamma_b$ for which $\left\lvert P_aP_b\right\rvert<1$

We have to consider three possibilities:

i) $\gamma_a$ and $\gamma_b$ are entirely in the exterior of each other.

In this case, their radii have a sum of $n-1$ and the distance of their centers is less than $n$, so there exist two appropriate points $P_a$ and $P_b$

ii) $\gamma_b$ is in the interior of $\gamma_a$.

Then $r(\gamma_b)=r(\gamma_a)-1$ and $u_0\neq u_n$, thus, $\gamma_b$ does not fit into the disk around $u_0$ with radius $r(\gamma_a)-1$, which means that there exists a point $P_b$ of $\gamma_b$ outside of this disk, which is closer to $\gamma_a$ than $1$, so we can take a point $P_a$ on $\gamma_a$ with $\left\lvert P_aP_b\right\rvert<1$.

iii) $\gamma_a$ and $\gamma_b$ intersect.

In this case, their intersection point can be used as both $P_a$ and $P_b$.

Other possibility cannot exist as $r(\gamma_b)\le r(\gamma_a)$, thus, $\gamma_a$ cannot be contained in the interior of $\gamma_b$ and $u_0\neq u_n$, thus $\gamma_a$ cannot coincide with $\gamma_b$.

$P_1\neq u_n$ and $P_2\neq u_0$, since $r(\gamma_a)\ge r(\gamma_b)\ge1>\left\lvert P_aP_b\right\rvert$ and for any point $P_1'\in\gamma_a\setminus u_n$ close enough to $P_a$, $\left\lvert P_a'P_b\right\rvert<1$. Also, since $\gamma_a\setminus u_n$ does not fit into an open disk of radius $1$ and it is continous, for any such $P_1'$, there exists a point $P_2'\in\gamma_b$ for which $\left\lvert P_a'P_b'\right\rvert=1$ because $\gamma_a\setminus u_n$ is continuous. So if we choose distinct $\left(P_a'\right)^{(i)}$ in $\gamma_a\cup B(P_2,1)$ for all $i\in\mathbb{Z}_{>0}$, then we can choose some $\left(P_b'\right)^{(i)}$ in $\gamma_b$, which has distance $1$ from it. And out of the infinitely many $\left(\left(P_a'\right)^{(i)},\left(P_b'\right)^{(i)}\right)$ pairs generated this way, $\left(P_a'\right)^{(i)}=\left(P_b'\right)^{(j)}$ cannot occur at all for $i\neq j$ and for a fixed $i$, $\left(P_b'\right)^{(i)}=\left(P_b'\right)^{(j)}$ can occur for at most one $j>i$, since there is at most one point of $\gamma_a\setminus \left(P_a'\right)^{(i)}$ with distance $1$ from $\left(P_b'\right)^{(i)}$, so that is the only possible $\left(P_a'\right)^{(j)}$. So we still can choose an infinite index set $J_0$ out of $\mathbb{Z}_{>0}$, for which neither $\left(P_a'\right)^{(i)}=\left(P_a'\right)^{(i)}$ nor $\left(P_b'\right)^{(i)}=\left(P_b'\right)^{(i)}$ can occur for any $i<j,i,j\in J_0$.

Now for all $j\in J_0$ define $u_i^{(j)}=u_0+\frac{i}{\left\lfloor\frac{n}{2}\right\rfloor}\cdot\vec{u_0\left(P_a'\right)^{(j)}}$ for $0\le i\le\left\lfloor\frac{n}{2}\right\rfloor$ and $u_i^{(j)}=u_n+\frac{n-i}{\left\lceil\frac{n}{2}\right\rceil-1}\cdot\vec{u_n\left(P_b'\right)^{(j)}}$ for $\left\lfloor\frac{n}{2}\right\rfloor+1\le i\le n$. Thus, $u_0^{(j)}=u_0$, $u_n^{(j)}=u_n$, $u_{\left\lfloor\frac{n}{2}\right\rfloor}^{(j)}=\left(P_a'\right)^{(j)}$ and $u_{\left\lfloor\frac{n}{2}\right\rfloor}^{(j)}=\left(P_b'\right)^{(j)}$.

For these, $\left\lvert u_{i}^{(j)}u_{i+1}^{(j)}\right\rvert=\frac{\left\lvert\vec{u_0\left(P_a'\right)^{(j)}}\right\rvert}{\left\lfloor\frac{n}{2}\right\rfloor}=1$ for $0\le i\le\left\lfloor\frac{n}{2}\right\rfloor-1$, $\left\lvert u_{i}^{(j)}u_{i+1}^{(j)}\right\rvert=\left\lvert\left(P_a'\right)^{(j)}\left(P_b'\right)^{(j)}\right\rvert=1$ for $i=\left\lfloor\frac{n}{2}\right\rfloor$ and $\left\lvert u_{i-1}^{(j)}u_i^{(j)}\right\rvert=\frac{\left\lvert\vec{u_n\left(P_b'\right)^{(j)}}\right\rvert}{\left\lceil\frac{n}{2}\right\rceil-1}=1$ for $\left\lfloor\frac{n}{2}\right\rfloor+1\le i\le n-1$. Thus, $\left\lvert u_i^{(j)}u_{i+1}^{(j)}\right\rvert=1$ holds for all $0\le i\le n-1$ and $j\in J_0$ so the $u_i^{(j)}$ ($i\in\left\lbrace0,1,...,n\right\rbrace$, $j\in J_0$) indeed form $(1,d)$-representations of $\left\lbrace u_0,u_1,u_2,...,u_n\right\rbrace$ with the possible exception of coincidences.

Also, since $\left\lvert\vec{u_0u_i^{(j)}}\right\rvert=\frac{i}{\left\lfloor\frac{n}{2}\right\rfloor}\cdot\vec{u_0\left(P_a'\right)^{(j)}}=i$ ($1\le i\le\left\lfloor\frac{n}{2}\right\rfloor$, $j\in J_0$) and $\left\lvert\vec{u_nu_i^{(j)}}\right\rvert=\frac{n-i}{\left\lfloor\frac{n}{2}\right\rfloor}\cdot\vec{u_n\left(P_b'\right)^{(j)}}=n-i$ ($\left\lfloor\frac{n}{2}\right\rfloor+1\le i\le n$, $j\in J_0$, so they also satisfy Properties 2 and 3 of the Lemma.

And $u_i^{(j)}=u_0+\frac{i}{\left\lfloor\frac{n}{2}\right\rfloor}\cdot\vec{u_0\left(P_a'\right)^{(j)}}\neq u_0+\frac{i}{\left\lfloor\frac{n}{2}\right\rfloor}\cdot\vec{u_0\left(P_a'\right)^{(k)}}=u_i^{(k)}$ for $1\le i\le\left\lfloor\frac{n}{2}\right\rfloor$, $j,k\in J_0$, and similarly, $u_i^{(j)}=u_n+\frac{n-i}{\left\lceil\frac{n}{2}\right\rceil-1}\cdot\vec{u_n\left(P_b'\right)^{(j)}}\neq u_0+\frac{n-i}{\left\lceil\frac{n}{2}\right\rceil-1}\cdot\vec{u_n\left(P_b'\right)^{(k)}}=u_i{(k)}$ for $\left\lfloor\frac{n}{2}\right\rfloor+1\le i\le n-1$, $j,k\in J_0$, so they also satisfy Property 1.

So the only question left is that for how many fixed $j$ can coincidences occur among the $u_i^{(j)}$.

For any $k\in J_0$, $u_i^{(k)}$ and $u_j^{(k)}$ cannot coincide if $i\neq j$ and $1\le i,j\le\left\lfloor\frac{n}{2}\right\rfloor$, because their distance from $u_0$ is different. Similarly, they cannot coincide if $i\neq j$ and $\left\lfloor\frac{n}{2}\right\rfloor+1\le i,j\le n$, since their distance from $u_n$ is different.  And if $1\le i\le\left\lfloor\frac{n}{2}\right\rfloor$ and $\left\lfloor\frac{n}{2}\right\rfloor+1\le j\le n-1$, $u_i^{(k)}$ and $u_j^{(k)}$ can coincide for at most two $k$, since this only can occur in the two intersection points of the circle around $u_0$ of radius $i$ and the circle around $u_n$ of radius $n-j$. Thus, in total, there are at most $2\cdot\left\lfloor\frac{n}{2}\right\rfloor\cdot\left(\left\lceil\frac{n}{2}\right\rceil-1\right)$ possible $k$ for which $u_0^{(k)}$, $u_1^{(k)}$, ..., $u_n^{(k)}$ do not form a $(1,d)$-representation of $\pi$. So if we take the set $J$ with these $k$s left out of $J_0$, it still will be an infinite set fulfilling the required properties.
\end{proof}

\begin{lemma}\label{lem:grid}
Suppose that for a red $n\times k$ grid (using the above notations) $u_{0,0}$, $u_{n,0}$, $u_{0,k}$ and $u_{n,k}$ are given so that they are distinct, $\vec{u_{0,0}u_{0,k}}=\vec{u_{n,0}u_{n,k}}$, $\left\lvert\vec{u_{0,0}u_{0,k}}\right\rvert=\left\lvert\vec{u_{n,0}u_{n,k}}\right\rvert<k$ and $\left\lvert\vec{u_{0,0}u_{n,0}}\right\rvert=\left\lvert\vec{u_{k,0}u_{n,k}}\right\rvert<n$. Also suppose that a finite set of points $\mathcal{P}\subseteq\mathbb{R}^2\setminus\left(u_{0,0}\cup u_{0,k}\cup u_{n,0}\cup u_{n,k}\right)$ is given. Then there exists a $(1,d)$-representation of the grid with $u_{0,0}$, $u_{0,k}$, $u_{n,0}$ and $u_{n,k}$ in the given locations and no vertices coincident with any of the points of $\mathcal{P}$.

The same statement applies for blue grids with the exception that instead of $\left\lvert\vec{u_{0,0}u_{0,k}}\right\rvert=\left\lvert\vec{u_{n,0}u_{n,k}}\right\rvert<k$.
\end{lemma}

\begin{proof}
Call the path $u_{0,0}u_{1,0}...u_{n,0}$ $\pi_0$ and the path $u_{0,k}u_{1,k}...u_{n,k}$ $\pi_k$.

Because of Lemma \ref{lem:path}, we can take an infinite index set $J$ and points $u_{i,0}^{(j)}$ for $1\le i\le n-1$ and $j\in J$ for which $u_{0,0}$, $u_{1,0}^{(j)}$, $u_{2,0}^{(j)}$ ... , $u_{n-1,0}^{(j)}$ and $u_{n,0}$ form a $(1,d)$-representation of $\pi_0$ satisfying the following properties:

1) $u_{i,0}^{(j)}\neq u_{i,0}^{(l)}$ for $i\in\left\lbrace 0,1,...,n\right\rbrace$, $j,l\in J$.

2) $\left\lvert u_{0,0}u_{i,0}^{(j)}\right\rvert=i$ for $0\le i\le\left\lfloor\frac{n}{2}\right\rfloor$, $j\in J$.

3) $\left\lvert u_{n,0}^{(j)}u_{i,0}^{(j)}\right\rvert=n-i$ for $\left\lfloor\frac{n}{2}\right\rfloor+1\le i\le n,j\in J$.

For any $j\in J$, call the representation generated this way $\pi_0^{(j)}$.

Now define points $u_{i,k}^{(j)}$ ($1\le i\le n-1$, $j\in J$) as $u_{i,0}^{(j)}+\vec{u_{0,0}u_{0,k}}$. They trivially have the same properties as the $u_{i,0}^{(j)}$s do. Also, call these representations $\pi_k^{(j)}$ for all $j\in J$. Because of $\pi_0^{(j)}$ and $\pi_k^{(j)}$ fulfilling properties 1)-3), the only way any two points of $\pi_0^{(j)}\cup\pi_k^{(j)}$ can be coincident is if they are from $\pi_0^{(j)}$ and $\pi_k^{(j)}$, respectively.

Now take a pair $\left(u_{i,0}^{(l)},u_{j,k}^{(l)}\right)$, where $1\le i,j\le n-1$ and $l\in J$. If $1\le i,j\le\left\lfloor\frac{n}{2}\right\rfloor$, then $u_{i,0}^{(l)}$ is on the circle around $u_{0,0}$ with radius $i$, while $u_{j,k}^{(l)}$ is on the circle around $u_{0,k}$ with radius $j$, so they only can meet in the two intersection points of these two circles (the two circles have different centers). Similarly if $1\le i\le\left\lfloor\frac{n}{2}\right\rfloor$ and $\left\lfloor\frac{n}{2}\right\rfloor+1\le j\le n-1$, then $u_{i,0}^{(l)}$ is on the circle around $u_{0,0}$ with radius $i$, while $u_{j,k}^{(l)}$ is on the circle around $u_{n,k}$ of radius $n-j$, if $\left\lfloor\frac{n}{2}\right\rfloor+1\le i\le n-1$ and $1\le j\le\left\lfloor\frac{n}{2}\right\rfloor$, then $u_{i,0}^{(l)}$ is on the circle around $u_{n,0}$ with radius $n-i$ and $u_{j,k}^{(l)}$ is on the circle around $u_{0,k}$ with radius $j$ and if $\left\lfloor\frac{n}{2}\right\rfloor+1\le i\le n-1$, then $u_{i,0}$ is on the circle around $u_{n,0}$ with radius $n-i$ and $u_{j,k}^{(l)}$ is on the circle around $u_{n,k}$ with radius $n-j$. Thus, because of property 1), they coincide for at most two $l$s regardless of the choice of $i$ and $j$.

Also, since the location of $u_{i,0}$ differs in all representations, for all $1\le i\le n-1$ and all $P\in\mathcal{P}$, $u_{i,0}=P$ can occur in at most one of the chosen representations. The same applies for the coincidence $u_{i,k}=P$ for some $1\le i\le n-1$ and some $P\in\mathcal{P}$.

Thus, there are at most $2\cdot(n-1)^2+2\cdot(n-1)\cdot\left\lvert\mathcal{P}\right\rvert$ numbers $j\in J$ for which $\left\lbrace u_{0,0}^{(j)},u_{1,0}^{(j)},...,u_{n,0}^{(j)}\right\rbrace,\left\lbrace u_{0,k}^{(j)},u_{1,k}^{(j)},...,u_{n,k}^{(j)}\right\rbrace$ and $\mathcal{P}$ are pairwise disjoint. Choose such a $j$ and fix the representation of $\left\lbrace u_{0,0},u_{1,0},...,u_{n,0}\right\rbrace$ and $\left\lbrace u_{0,k},u_{1,k},...,u_{n,k}\right\rbrace$ induced by $j$: for all $1\le i\le n-1$, let $u_{i,0}=u_{i,0}^{(j)}$ and $u_{i,k}=u_{i,k}^{(j)}$.

Now again use Lemma \ref{lem:path}, but for the $k$-path $u_{0,0}u_{0,1}...u_{0,k}$. This way we get $(1,d)$-representations of $\left\lbrace u_{0,0}u_{0,1}...u_{0,k}\right\rbrace$ called $\left\lbrace u_{0,0}^{(j)}u_{0,1}^{(j)}...u_{0,k}^{(j)}\right\rbrace$ for $j\in J'$, where $J'$ is some infinite index set fulfilling the following properties:

1) $u_{0,i}^{(j)}\neq u_{0,i}^{(l)}$ for $i\in\left\lbrace0,1,...,n\right\rbrace$, $j,l\in J'$.

2) $\left\lvert u_{0,0}u_{0,i}^{(j)}\right\rvert=i$ for $0\le i\le\left\lfloor\frac{k}{2}\right\rfloor$, $j\in J'$

3) $\left\lvert u_{0,k}u_{0,i}^{(j)}\right\rvert=i$ for $\left\lfloor\frac{n}{2}\right\rfloor+1\le i\le n$, $j\in J'$.

Now define $u_{i,l}^{(j)}$ as $u_{i,0}+\vec{u_{0,0}u_{0,l}^{(j)}}$ for all $1\le i\le n$, $1\le l\le k-1$ and $j\in J'$. These give a $(1,d)$-representation of the grid apart from the possible coincidences: for any $1\le i\le n-1$ and $1\le l\le k$, $\left\lvert u_{i+1,l}^{(j)}-u_{i,l}^{(j)}\right\rvert=\left\lvert\left(u_{i+1,0}+\vec{u_{0,0}u_{0,l}^{(j)}}\right)-\left(u_{i,0}+\vec{u_{0,0}u_{0,l}^{(j)}}\right)\right\rvert=\left\lvert u_{i+1,0}-u_{i,0}\right\rvert=1$ and for any $1\le i\le n$ and $1\le l\le k-1$,
$$\left\lvert u_{i,l+1}^{(j)}-u_{i,l}^{(j)}\right\rvert=\left\lvert\left(u_{i,0}+\vec{u_{0,0}u_{0,l+1}^{(j)}}\right)-\left(u_{i,0}+\vec{u_{0,0}u_{0,l}^{(j)}}\right)\right\rvert=\left\lvert u_{0,l+1}^{(j)}-u_{0,l}^{(j)}\right\rvert=1.$$
Suppose that for some $j\in J'$, $0\le i,m\le n$ and $0\le l,m\le k$, $u_{i,l}^{(j)}$ and $u_{m,o}^{(j)}$ coincide (but they are not the same vertex, i. e. the ordered pair $\left(i,l\right)$ differs from the ordered pair $\left(m,o\right)$). If $i=l$, then this cannot happen as $\left\lbrace u_{i,0}^{(j)},u_{i,1}^{(j)},...,u_{i,n}^{(j)}\right\rbrace$ is a $(1,d)$-representation of $\left\lbrace u_{i,0},u_{i,1},...,u_{i,n}\right\rbrace$. If $i\neq l$, then $u_{i,l}^{(j)}$ is a point of the circle of radius $l$ around $u_{i,0}$ (in case $0\le l\left\lfloor\frac{k}{2}\right\rfloor$) or the circle of radius $k-l$ around $u_{i,k}$ (in case $\left\lfloor\frac{k}{2}\right\rfloor+1\le l\le k$) and similarly, $u_{m,o}$ is a point of the circle of radius $o$ around $u_{m,0}$ (in case $0\le o\left\lfloor\frac{k}{2}\right\rfloor$) or the circle of radius $k-o$ around $u_{m,k}$ (in case $\left\lfloor\frac{k}{2}\right\rfloor+1\le o\le k$). Thus, since they are points of some fixed circles with different centers, they can coincide for at most two $j\in J'$. So in total, this means that such a coincidence can occur in at most $2\cdot\left(n+1\right)\cdot k\cdot n\cdot k$ elements of $J'$. And since $J$ is infinite, this means that we can choose a $j\in J'$ for which $u_{0,0},u_{1,0},...,u_{n,0}$, $u_{0,k},u_{1,k},...,u_{n,k}$ and $u_{i,l}^{(j)}$ ($0\le i\le n$, $1\le l\le k-1$) give a $(1,d)$-representation of the grid.
\end{proof}

\begin{definition}
Take two pairs of vertices in an EBG $G$ and let connecting two ordered pairs of vertices ($(a,b)$ and $(c,d)$) with a red or blue $n\times k$ grid mean that we insert a red or blue $n\times k$ grid on these four vertices so that $a$ corresponds to $u_{0,0}$, $b$ correspond to $u_{n,0}$, $c$ correspond to $u_{0,k}$ and $d$ correspond to $u_{n,k}$.
\end{definition}

Define the radius $r(S)$ of a finite set $S$ of points as the minimal radius of any disk covering $S$ and the center $O(S)$ as the center of this disk.

\begin{lemma}\label{prop:gridconnection}
Take finite EBGs $G_1$, $G_2$, ..., $G_m$ for some $m\in\mathbb{Z}^+$. Define a graph $G$ in the following way:

Take vertices $t_i,u_i,v_i,w_i$ ($1\le i\le l$) for some $l\in\mathbb{Z}^+$ so that for all $(1\le i\le l)$ $t_i$ and $u_i$ are in the same $G_j$ and $v_i$ and $w_i$ are in the same $G_o$ ($j=o$ is allowed). Then for all $1\le i\le l$, draw a red $n_i\times k_i$ grid $R_i$ for some $n_i,k_i\in\mathbb{Z}^+ (1\le i\le l)$ with corners $t_i$, $u_i$, $w_i$ and $v_i$ in this order.

Now suppose that for some $d$, there exist $(1,d)$-representations of $G_1$, $G_2$, ... $G_m$ fulfilling the following conditions:

1) For all $1\le i\le l$, $\vec{t_iu_i}=\vec{v_iw_i}$.

2) For all $1\le i\le l$, $\left(\left\lvert\vec{t_iu_i}\right\rvert=\right)\left\lvert\vec{v_iw_i}\right\rvert<n_i$.

3) For all $1\le i\le l$, if $t_i$ and $u_i$ are vertices of $G_j$, while $v_i$ and $w_i$ are vertices of $G_o$, then $k_i> r\left(V\left(G_j\right)\right)+r\left(V\left(G_o\right)\right)$.

Then there exists a $(1,d)$-representation of $G'$.
\end{lemma}

\begin{proof}
We will draw such a representation.

First, draw the $(1,d)$-representation of $G_1$ mentioned in the assumption of the lemma.

Now define $\delta$ as the minimum of $k_i-r\left(V\left(G_j\right)\right)+r\left(V\left(G_o\right)\right)$ for $\left(1\le i\le l\right)$ for which $t_i,u_i\in V(G_j)$ and $v_i,w_i\in V(G_o)$. From condition 3), this value is positive.

Now place $O_i$ ($2\le i\le m$) one after another so that they are all inside the open disk of radius $\frac{\delta}{2}$ around $O_1$ and if we translate $G_i$ along with $O_i$, no vertices of $G_i$ coincide with those of any previously placed $G_j$.

Now because of the assumptions, we can apply Lemma \ref{lem:grid} to draw the $R_j$ ($1\le j\le l$) one after another: the conditions for its endpoints are satisfied, and other then finding a $(1,d)$-representation, we only have to take care of $V\left(R_j\right)$ being disjoint from the finitely many vertices drawn previously.
\end{proof}

Proposition \ref{prop:pvebgtoebg} is a trivial consequence of Proposition \ref{prop:gridconnection}.
\end{proof}

From now on, connecting two green edges by a red grid will mean that they are in the same red equivalence class, while connecting two green edges by a blue grid will mean that they are in the same blue equivalence class, but these grids only will be virtual in the case of a PVEBG.

\vskip.5cm

\hypertarget{pf:main}{}\subsection{Proof of Proposition \ref{prop:main}}

First, we suppose that $p$ has at least one positive coefficient, otherwise $S_0\left(p,0,+\infty\right)=\emptyset$ because of the negative leading coefficient, thus a red $K_4$ is suitable for $G_0\left(p\right)$.

Now define the graph $A_1$ in the following way:

\begin{center}
\begin{minipage}{.32\textwidth}
	\centering
	\usetikzlibrary{arrows}
\definecolor{qqzzqq}{rgb}{0,0.6,0}
\definecolor{qqqqff}{rgb}{0,0,1}
\definecolor{ffqqqq}{rgb}{1,0,0}
\begin{tikzpicture}[line cap=round,line join=round]
\clip(-0.36,-0.48) rectangle (3.22,5.86);
\draw [color=ffqqqq] (0,0)-- (0.96,0.19);
\draw [color=ffqqqq] (0.96,0.19)-- (1.77,0.73);
\draw [color=qqqqff] (0,5)-- (0,0);
\draw [shift={(0.39,0.54)},color=qqzzqq,-latex]  plot[domain=4.09:5.73,variable=\t]({1*0.66*cos(\t r)+0*0.66*sin(\t r)},{0*0.66*cos(\t r)+1*0.66*sin(\t r)});
\draw [shift={(1.11,0.84)},color=qqzzqq,-latex]  plot[domain=4.48:6.12,variable=\t]({1*0.66*cos(\t r)+0*0.66*sin(\t r)},{0*0.66*cos(\t r)+1*0.66*sin(\t r)});
\draw [shift={(1.66,1.39)},color=qqzzqq,-latex]  plot[domain=-1.41:0.23,variable=\t]({1*0.66*cos(\t r)+0*0.66*sin(\t r)},{0*0.66*cos(\t r)+1*0.66*sin(\t r)});
\draw [shift={(1.96,2.11)},color=qqzzqq,-latex]  plot[domain=-1.02:0.63,variable=\t]({1*0.66*cos(\t r)+0*0.66*sin(\t r)},{0*0.66*cos(\t r)+1*0.66*sin(\t r)});
\draw [shift={(1.96,2.89)},color=qqzzqq,-latex]  plot[domain=-0.63:1.02,variable=\t]({1*0.66*cos(\t r)+0*0.66*sin(\t r)},{0*0.66*cos(\t r)+1*0.66*sin(\t r)});
\draw [shift={(1.66,3.61)},color=qqzzqq,-latex]  plot[domain=-0.23:1.41,variable=\t]({1*0.66*cos(\t r)+0*0.66*sin(\t r)},{0*0.66*cos(\t r)+1*0.66*sin(\t r)});
\draw [shift={(1.11,4.16)},dash pattern=on 2pt off 2pt,color=qqzzqq]  plot[domain=0.16:1.81,variable=\t]({1*0.66*cos(\t r)+0*0.66*sin(\t r)},{0*0.66*cos(\t r)+1*0.66*sin(\t r)});
\draw [shift={(0.39,4.46)},color=qqzzqq,-latex]  plot[domain=0.55:2.2,variable=\t]({1*0.66*cos(\t r)+0*0.66*sin(\t r)},{0*0.66*cos(\t r)+1*0.66*sin(\t r)});
\begin{scriptsize}
\fill [color=black] (0,0) circle (1.5pt);
\draw[color=black] (0.27,0.2) node {$a_{1,0}$};
\fill [color=black] (0,5) circle (1.5pt);
\draw[color=black] (0.23,5.2) node {$a_{1,2N}$};
\fill [color=black] (0.96,0.19) circle (1.5pt);
\draw[color=black] (0.9,0.4) node {$a_{1,1}$};
\fill [color=black] (1.77,0.73) circle (1.5pt);
\draw[color=black] (1.75,0.87) node {$a_{1,2}$};
\fill [color=black] (2.31,1.54) circle (1.5pt);
\draw[color=black] (2.15,1.68) node {$a_{1,3}$};
\fill [color=black] (2.5,2.5) circle (1.5pt);
\draw[color=black] (2.15,2.64) node {$a_{1,4}$};
\fill [color=black] (2.31,3.46) circle (1.5pt);
\draw[color=black] (1.9,3.5) node {$a_{1,5}$};
\fill [color=black] (1.77,4.27) circle (1.5pt);
\draw[color=black] (1.5,4.1) node {$a_{1,6}$};
\fill [color=black] (0.96,4.81) circle (1.5pt);
\draw[color=black] (1.4,4.95) node {$a_{1,2N-1}$};
\draw[color=red] (0.56,0) node {0};
\draw[color=red] (1.5,0.4) node {1};
\draw[color=red] (2.4,1.17) node {0};
\draw[color=red] (2.71,2.12) node {1};
\draw[color=red] (2.71,3.16) node {0};
\draw[color=red] (2.32,4.13) node {1};
\draw[color=red] (0.62,5.26) node {1};
\end{scriptsize}
\end{tikzpicture}
	
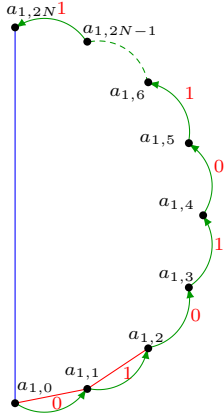
\captionof{figure}{Graph $A_1$. The dashed edge only represents that typically there there is more than one edge between $a_{1,6}$ and $a_{1,2N-1}$}
	\label{rrrbcycle6}
\end{minipage}
\begin{minipage}{.32\textwidth}
	\centering
	\usetikzlibrary{arrows}
\definecolor{qqzzqq}{rgb}{0,0.6,0}
\definecolor{ffqqqq}{rgb}{1,0,0}
\definecolor{uququq}{rgb}{0.25,0.25,0.25}
\definecolor{qqqqff}{rgb}{0,0,1}
\begin{tikzpicture}[line cap=round,line join=round]
\clip(-0.5,-0.4) rectangle (1.75,5.86);
\draw [color=qqqqff] (0,5)-- (0,0);
\draw [color=ffqqqq] (0,0)-- (0.75,0.63);
\draw [color=ffqqqq] (0.75,0.63)-- (0,1.25);
\draw [shift={(0.08,0.66)},color=qqzzqq,-Latex]  plot[domain=4.58:6.23,variable=\t]({1*0.66*cos(\t r)+0*0.66*sin(\t r)},{0*0.66*cos(\t r)+1*0.66*sin(\t r)});
\draw [shift={(0.08,1.91)},color=qqzzqq,-Latex]  plot[domain=4.58:6.23,variable=\t]({1*0.66*cos(\t r)+0*0.66*sin(\t r)},{0*0.66*cos(\t r)+1*0.66*sin(\t r)});
\draw [shift={(0.08,3.16)},color=qqzzqq,-Latex]  plot[domain=4.58:6.23,variable=\t]({1*0.66*cos(\t r)+0*0.66*sin(\t r)},{0*0.66*cos(\t r)+1*0.66*sin(\t r)});
\draw [shift={(0.08,4.41)},dash pattern=on 2pt off 2pt,color=qqzzqq]  plot[domain=4.58:6.23,variable=\t]({1*0.66*cos(\t r)+0*0.66*sin(\t r)},{0*0.66*cos(\t r)+1*0.66*sin(\t r)});
\draw [shift={(0.08,0.59)},color=qqzzqq,-Latex]  plot[domain=0.05:1.7,variable=\t]({1*0.66*cos(\t r)+0*0.66*sin(\t r)},{0*0.66*cos(\t r)+1*0.66*sin(\t r)});
\draw [shift={(0.08,1.84)},color=qqzzqq,-Latex]  plot[domain=0.05:1.7,variable=\t]({1*0.66*cos(\t r)+0*0.66*sin(\t r)},{0*0.66*cos(\t r)+1*0.66*sin(\t r)});
\draw [shift={(0.08,3.09)},color=qqzzqq,-Latex]  plot[domain=0.05:1.7,variable=\t]({1*0.66*cos(\t r)+0*0.66*sin(\t r)},{0*0.66*cos(\t r)+1*0.66*sin(\t r)});
\draw [shift={(0.08,4.34)},color=qqzzqq,-Latex]  plot[domain=0.05:1.7,variable=\t]({1*0.66*cos(\t r)+0*0.66*sin(\t r)},{0*0.66*cos(\t r)+1*0.66*sin(\t r)});
\begin{scriptsize}
\fill [color=black] (0,0) circle (1.5pt);
\draw[color=black] (-0.25,0.1) node {$a_{1,0}$};
\fill [color=black] (0,5) circle (1.5pt);
\draw[color=black] (-0.15,5.15) node {$a_{1,2N}$};
\fill [color=black] (0.75,0.63) circle (1.5pt);
\draw[color=black] (1.1,0.71) node {$a_{1,1}$};
\fill [color=black] (0,1.25) circle (1.5pt);
\draw[color=black] (-0.25,1.35) node {$a_{1,2}$};
\fill [color=black] (0.75,1.88) circle (1.5pt);
\draw[color=black] (1.1,1.96) node {$a_{1,3}$};
\fill [color=black] (0,2.5) circle (1.5pt);
\draw[color=black] (-0.25,2.6) node {$a_{1,4}$};
\fill [color=black] (0.75,3.13) circle (1.5pt);
\draw[color=black] (1.1,3.21) node {$a_{1,5}$};
\fill [color=black] (0,3.75) circle (1.5pt);
\draw[color=black] (-0.25,3.85) node {$a_{1,6}$};
\fill [color=black] (0.75,4.38) circle (1.5pt);
\draw[color=black] (1.25,4.46) node {$a_{1,2N-1}$};
\draw[color=red] (0.7,0.2) node {0};
\draw[color=red] (0.7,1.45) node {0};
\draw[color=red] (0.7,2.7) node {0};
\draw[color=red] (0.58,1.19) node {1};
\draw[color=red] (0.58,2.44) node {1};
\draw[color=red] (0.58,3.69) node {1};
\draw[color=red] (0.58,4.94) node {1};
\end{scriptsize}
\end{tikzpicture}
	
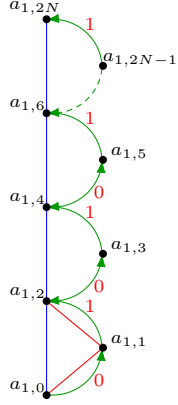
\captionof{figure}{The only possible $(1,d)$-representation of $A_1$ up to isometry for some $d$}
	\label{rrrbcycle7}
\end{minipage}
\begin{minipage}{.32\textwidth}
	\centering
	\usetikzlibrary{arrows}
\definecolor{qqffqq}{rgb}{0,1,0}
\definecolor{qqzzqq}{rgb}{0,0.6,0}
\definecolor{ffqqqq}{rgb}{1,0,0}
\definecolor{uququq}{rgb}{0.25,0.25,0.25}
\definecolor{qqqqff}{rgb}{0,0,1}
\begin{tikzpicture}[line cap=round,line join=round]
\clip(-0.5,-0.4) rectangle (1.75,5.86);
\draw [color=qqqqff] (0,5)-- (0,0);
\draw [color=ffqqqq] (0,0)-- (0.75,0.63);
\draw [color=ffqqqq] (0.75,0.63)-- (0,1.25);
\draw [shift={(0.08,0.66)},color=qqzzqq,-Latex]  plot[domain=4.58:6.23,variable=\t]({1*0.66*cos(\t r)+0*0.66*sin(\t r)},{0*0.66*cos(\t r)+1*0.66*sin(\t r)});
\draw [shift={(0.08,1.91)},color=qqzzqq,-Latex]  plot[domain=4.58:6.23,variable=\t]({1*0.66*cos(\t r)+0*0.66*sin(\t r)},{0*0.66*cos(\t r)+1*0.66*sin(\t r)});
\draw [shift={(0.08,3.16)},color=qqzzqq,-Latex]  plot[domain=4.58:6.23,variable=\t]({1*0.66*cos(\t r)+0*0.66*sin(\t r)},{0*0.66*cos(\t r)+1*0.66*sin(\t r)});
\draw [shift={(0.08,4.41)},dash pattern=on 2pt off 2pt,color=qqzzqq]  plot[domain=4.58:6.23,variable=\t]({1*0.66*cos(\t r)+0*0.66*sin(\t r)},{0*0.66*cos(\t r)+1*0.66*sin(\t r)});
\draw [shift={(0.08,0.59)},color=qqzzqq,-Latex]  plot[domain=0.05:1.7,variable=\t]({1*0.66*cos(\t r)+0*0.66*sin(\t r)},{0*0.66*cos(\t r)+1*0.66*sin(\t r)});
\draw [shift={(0.08,1.84)},color=qqzzqq,-Latex]  plot[domain=0.05:1.7,variable=\t]({1*0.66*cos(\t r)+0*0.66*sin(\t r)},{0*0.66*cos(\t r)+1*0.66*sin(\t r)});
\draw [shift={(0.08,3.09)},color=qqzzqq,-Latex]  plot[domain=0.05:1.7,variable=\t]({1*0.66*cos(\t r)+0*0.66*sin(\t r)},{0*0.66*cos(\t r)+1*0.66*sin(\t r)});
\draw [shift={(0.08,4.34)},color=qqzzqq,-Latex]  plot[domain=0.05:1.7,variable=\t]({1*0.66*cos(\t r)+0*0.66*sin(\t r)},{0*0.66*cos(\t r)+1*0.66*sin(\t r)});
\draw [color=qqffqq,Latex-] (1.23,-0.22)-- (0.75,0.63);
\begin{scriptsize}
\fill [color=black] (0,0) circle (1.5pt);
\draw[color=black] (-0.25,0.1) node {$a_{j,0}$};
\fill [color=black] (0,5) circle (1.5pt);
\draw[color=black] (-0.15,5.15) node {$a_{j,2N}$};
\fill [color=black] (0.75,0.63) circle (1.5pt);
\draw[color=black] (1.1,0.71) node {$a_{j,1}$};
\fill [color=black] (0,1.25) circle (1.5pt);
\draw[color=black] (-0.25,1.35) node {$a_{j,2}$};
\fill [color=black] (0.75,1.88) circle (1.5pt);
\draw[color=black] (1.1,1.96) node {$a_{j,3}$};
\fill [color=black] (0,2.5) circle (1.5pt);
\draw[color=black] (-0.25,2.6) node {$a_{j,4}$};
\fill [color=black] (0.75,3.13) circle (1.5pt);
\draw[color=black] (1.1,3.21) node {$a_{j,5}$};
\fill [color=black] (0,3.75) circle (1.5pt);
\draw[color=black] (-0.25,3.85) node {$a_{j,6}$};
\fill [color=black] (0.75,4.38) circle (1.5pt);
\draw[color=black] (1.25,4.46) node {$a_{j,2N-1}$};
\draw[color=ffqqqq] (0.7,0.2) node {$j-1$};
\draw[color=ffqqqq] (1,1.45) node {$j-1$};
\draw[color=ffqqqq] (1,2.7) node {$j-1$};
\draw[color=ffqqqq] (0.58,1.22) node {$j$};
\draw[color=ffqqqq] (0.58,2.47) node {$j$};
\draw[color=ffqqqq] (0.58,3.72) node {$j$};
\draw[color=ffqqqq] (0.58,4.97) node {$j$};
\draw[color=ffqqqq] (1.35,0.22) node {$j-2$};
\fill [color=black] (1.23,-0.22) circle (1.5pt);
\draw[color=black] (1.12,0) node {$a_{j,2N+1}$};
\end{scriptsize}
\end{tikzpicture}
	
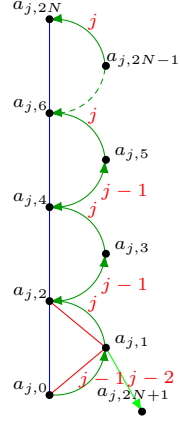
\captionof{figure}{The only possible $(1,d)$-representation of $A_j$ up to isometry for some $d$}
	\label{rrrbcycle8}
\end{minipage}
\end{center}

Let $M$ be the maximal root of $p$. Then let $N$ be an integer larger than $\frac{M}{2\sin\frac{\pi}{2\cdot deg(p)}}$.

Let the vertices of $A_1$ be $a_{1,i}$, where $i=0,1,...,2N$, connect $a_{1,0}$, $a_{1,1}$ and $a_{1,2}$ with red edges in this order and connect $(a_{1,i},a_{1,i+1})$ and $(a_{1,i+2},a_{1,i+3})$ with a red grid for $0\le i\le 2N-3$.

From this, we get that the only way $A_1$ can have a $(1,d)$-representation is similarly to the one drawn in Figure \ref{rrrbcycle7}: with $N$ isometric isosceles triangles with base length $\frac{d}{N}$ and with an apex angle we will call $\varphi$. So if we draw $A_1$ on the complex plane and we denote $\vec{a_{1,0}a_{1,1}}$ by $1$, $\vec{a_{1,1}a_{1,2}}$ by $\varepsilon$ and $\vec{a_{1,0}a_{1,2N}}$ by $z$, then $\left\lvert\varepsilon\right\rvert=1$ and $z=N\cdot\left(1+\varepsilon\right)$.

Now suppose WLOG that $\varepsilon$ has a positive imaginary part (if not, then we can draw the complex plane unconventionally, reflected to the real axis). In this case, $\varepsilon$ is uniquely defined for a fixed $d$: $\left\lvert\varepsilon\right\rvert=1$ and $Arg(\varepsilon)=\pi-\varphi$ (by Arg, we always will mean its value in $\left(-\pi,\pi\right]$).

Now define graph $A_j$ for $2\le j\le deg(p)$ in the following way:

Take an isomorphic copy of $A_1$, where the equivalent of $a_{1,i}$ is called $a_{j,i}$ (for all $i$) and call it $A_j^{-}$. Add an extra isolated vertex called $a_{j,2N+1}$ to $A_j^{-}$ and call the resulting graph $A_j$.

Similarly as above, $A_j^{-}$ is isomorphic to the drawing of $A_1$ in Figure \ref{rrrbcycle7} for $2\le j\le deg(p)$.

Now generate $A$ from the $A_j$s in the following way:

Connect $\left(a_{j,1},a_{j,2}\right)$ and $\left(a_{j+1,0},a_{j+1,1}\right)$ ($0\le j\le deg(p)-1$) via a red grid. Also, connect $\left(a_{j,0},a_{j,1}\right)$ and $\left(a_{j+1,1},a_{j+1,2N+1}\right)$ ($0\le j\le deg(p)-1$) via a red grid.

\begin{lemma}\label{lem:a}
In all $(1,d)$-representations of $A$, $\vec{a_{j,0}a_{j,1}}=\varepsilon^{j-1}$ and $\vec{a_{j,1}a_{j,2}}=\varepsilon^j$ for $1\le j\le deg(p)$.\hfill\hyperlink{pf:a}{Proof}
\end{lemma}

\begin{proof}
We will prove it by induction:

For $j=1$ it is true by the definition of $\varepsilon$.

If $j>1$ and the statement is true for $j-1$, then it is also true for $j$, since $\vec{a_{j,0}a_{j,1}}=\vec{a_{j-1,1}a_{j-1,2}}=\varepsilon^{j-1}$ and because of the isomorphy of $A_{j-1}$ and $A_j$, this only leaves the possibilities that $\vec{a_{j,1}a_{j,2}}=\varepsilon^{j-2}$ and $\vec{a_{j,1}a_{j,2}}=\varepsilon^j$. But since $\vec{a_{j,1}a_{j,N+1}}=\vec{a_{j-1,0}a_{j-1,1}}=\varepsilon^{j-2}$, the latter would mean that $a_{j,2}$ and $a_{j,N+1}$ coincide, leading us to a contradiction.
\end{proof}

\begin{lemma}\label{lem:a2}
If $d\le M$, then $A$ has a $(1,d)$-representation (in fact, the upper bound $M$ is typically not sharp).
\end{lemma}

\begin{proof}
$A_1$ has a $(1,d)$-representation: the one in Figure \ref{rrrbcycle7} exists if $N$ is even and $N>d\ge M$, which follows from the definition of $N$.

Similarly, we can take the $(1,d)$-representations of the $A_j^-$s, but which are rotated so that $\vec{a_{j,i}a_{j,i+1}}=\varepsilon^{j-1}$ for even $i$ and $\vec{a_{j,i+1}a_{j,i+2}}=\varepsilon$ for odd $i$ as described above, and these are indeed $(1,d)$-representations of the $A_j^-$s. And in the above, we have proved that the only possible position of $a_{j,N+1}$ is $a_{j,1}+\varepsilon^{j-2}$, which is the reflection of $a_{j,2}$ to $a_{j,1}$, and is thus, on the opposite side of the line $a_{j,0}a_{j,1}$, than the vertices of $A_j^{-}$, so it cannot coincide with any of them. So if we draw the $A_j$s as in \ref{rrrbcycle8}, then from Proposition \ref{prop:gridconnection}, we get that $A$ indeed has a $(1,d)$-representation. And since all the $A_j$ have a bounded diameter (at most $1+M$), the usage of grids was allowed.
\end{proof}

\begin{lemma}\label{lem:argepsilon}
If $d\le M$, then $0=Arg\left(1\right)>Arg\left(\varepsilon^2\right)>Arg\left(\varepsilon^4\right)>...>Arg\left(\varepsilon^{deg(p)}\right)>-\pi$.
\end{lemma}

\begin{proof}
Since $a_{1,0}a_{1,1}a_{1,2}$ is an isosceles triangle with base length at most $\frac{M}{N}<2\sin{\frac{\pi}{2\cdot deg(p)}}$, $\left\lvert\frac{arg\left(-\varepsilon\right)}{2}\right\rvert<\frac{\pi}{2\cdot deg(p)}\Rightarrow\left\lvert\arg\left(-\varepsilon\right)\right\rvert<\frac{\pi}{deg(p)}$, from which the statement follows.
\end{proof}

\begin{lemma}\label{lem:qpqn}
If we define $p_p(d)$ as the polynomial formed by the positive terms of $p(d)$ and $p_n(d)$ as the polynomial formed by the negative terms of $p(d)$, then both $p_p(d)\cdot\varepsilon^{\frac{deg(p_p)}{2}}$ and $p_n(d)\cdot\varepsilon^{\frac{deg(p_n)}{2}}$ can be written as a polynomial of $\varepsilon$, call them $q_p(d)=\sum\limits_{i=0}^{deg(q_p)}{\beta_i\cdot\varepsilon^i}$ and $q_n(d)=\sum\limits_{i=0}^{deg(q_n)}{\gamma_i\cdot\varepsilon^i}$, respectively. For these, $deg(q_p)=deg(p_p)$, $deg(q_n)=deg(p_n)$ and all their coefficients are positive.
\end{lemma}

\begin{proof}
First, we write any power of $z$ with a positive integer exponent as the polynomial of $\varepsilon$:

$$z^n=\left(N\cdot(1+\varepsilon)\right)^n=\sum\limits_{i=0}^n{N^n\cdot\binom{n}{i}\cdot\varepsilon^i}$$

And since $Arg(z)=\angle{a_{1,1}a_{1,0}a_{1,N}}=\frac{\pi-\varphi}{2}=\frac{Arg\left(\varepsilon\right)}{2}$, we can write $d^{2i}\cdot\varepsilon^{\frac{deg(p_p)}{2}}$ ($i=0,1,2,...,\frac{deg(p_p)}{2}$) as a polynomial of $\varepsilon$:

$d^{2i}\cdot\varepsilon^{\frac{deg(p_p)}{2}}=\sum\limits_{j=0}^{2i}{N^{2i}\cdot\binom{2i}{j}\cdot\varepsilon^{\frac{deg(p_p)}{2}-i+j}}$

Now by summing the terms of $p_p(d)^\cdot\varepsilon^{\frac{deg(p_p)}{2}}$ up, we get

$q_p(\varepsilon)=p_p(d)\cdot\varepsilon^{\frac{deg(p_p)}{2}}=\sum\limits_{i=0,\alpha_{2i}>0}^{\frac{deg(p_p)}{2}}{\alpha_{2i}\cdot d^i\cdot\varepsilon^{\frac{deg(p_p)}{2}}}=\sum\limits_{i=0,\alpha_{2i}>0}^{\frac{deg(p_p)}{2}}{\alpha_{2i}\cdot\sum\limits_{j=0}^{2i}\cdot\binom{2i}{j}\cdot\varepsilon^{\frac{deg(p_p)}{2}-i+j}}=$

$=\sum\limits_{k=0}^{deg(p_p)}{\sum\limits_{i=\left\lvert\frac{deg(p_p)}{2}-i\right\rvert,\alpha_{2i}>0}^{\frac{deg(p_p)}{2}}{\left(\alpha_{2i}\cdot N^{2i}\cdot\binom{2i}{i+k-\frac{deg(p_p)}{2}}\right)}\cdot\varepsilon^k}$.

And since all the summands in this double sum are positive and for all $k$, at least one such summand exists as $\alpha_{deg(p_p)}$ is positive and $i=\frac{deg(p)}{2}$ occurs for all $k$.

$q_n(\varepsilon)$ is similarly defined as $\sum\limits_{k=0}^{deg(p_n)}{\sum\limits_{i=\left\lvert\frac{deg(p_n)}{2}-i\right\rvert,\alpha_{2i}>0}^{\frac{deg(p_n)}{2}}{\left(\alpha_{2i}\cdot N^{2i}\cdot\binom{2i}{i+k-\frac{deg(p_p)}{2}}\right)}\cdot\varepsilon^k}$ and the coefficients of $q_n$ are also all positive.
\end{proof}

Now our plan is as follows:

We take three copies of $A$ and with appropriate connections between these and a graph $B$ to be defined, we guarantee that some three vertices of $B$ have distances $\left\lvert q_p(\varepsilon)\right\rvert=p_p(d)$, $\left\lvert q_p(\varepsilon)\right\rvert=p_p(d)$ and $2\left\lvert q_n(\varepsilon)\right\rvert=2p_n(d)$. This graph only can have a $(1,d)$-representation if $p_p(d)\ge p_n(d)\Leftrightarrow p(d)\ge 0\Leftrightarrow d\in S_0\left(p,0,+\infty\right)$. It is simple from here to construct such a graph, but if we are not careful, some unwanted coincidences among the vertices can occur, meaning that not all of $S_0\left(p,0,+\infty\right)$ will be inside the range of $G(p)$. Hence a rather complicated construction follows.

We will have two separate cases:

1) $deg(q_p)=0$ and $deg(q_n)\ge2$.

2) $deg(q_p)\ge2$ and $deg(q_n)\ge2$.

Define graph $B$ in the following way (the vertices of $B_i$ ($1\le i\le 10$) are always called $b_{i,j}$ ($0\le j\le\left\lvert V(B_i)\right\rvert-1$)):

In case 1, the numbers of vertices of the $B_i$ are as follows:

$\left\lvert V(B_1)\right\rvert=\left\lvert V(B_4)\right\rvert=\sum\limits_{i=0}^{deg(q_p)}{\beta_i+1}$.

$\left\lvert V(B_2)\right\rvert=\left\lvert V(B_3)\right\rvert=\left\lvert V(B_5)\right\rvert=\left\lvert V(B_6)\right\rvert=0$.

$\left\lvert V(B_7)\right\rvert=2\cdot\sum\limits_{i=0}^{\frac{deg(q_n)}{2}}{\beta_{2i}+1}$.

$\left\lvert V(B_8)\right\rvert=2\cdot\sum\limits_{i=0}^{\frac{deg(q_n)}{2}}{\beta_{2i+1}+1}$.

$\left\lvert V(B_9)\right\rvert=\left\lvert V(B_{10})\right\rvert=3$.

While the green edges are located in the way drawn in Figure \ref{B6}.

\begin{figure}
	\centering
\pagestyle{empty}
\definecolor{qqzzqq}{rgb}{0,0.6,0}
\begin{tikzpicture}[line cap=round,line join=round,scale=0.5]
\clip(-7.2,-3.84) rectangle (14,6.42);
\draw [color=qqzzqq] (5,-3.22)-- (6.24,-0.32);
\draw [color=qqzzqq] (6.58,1.76)-- (6.68,4.92);
\draw [color=qqzzqq] (-0.08,-2.62)-- (-3.72,1.02);
\draw [color=qqzzqq] (-3.32,3.16)-- (1.66,5.58);
\draw [color=qqzzqq] (1.68,-1.72)-- (-1.96,1.92);
\draw [color=qqzzqq] (-1.96,1.92)-- (3.02,4.34);
\draw [color=qqzzqq] (3.44,-1.4)-- (4.68,1.5);
\draw [color=qqzzqq] (4.68,1.5)-- (4.78,4.66);
\draw [color=qqzzqq] (3.44,-1.4)-- (4.78,4.66);
\draw [color=qqzzqq] (1.68,-1.72)-- (3.02,4.34);
\begin{scriptsize}
\fill [color=black] (5,-3.22) circle (1.5pt);
\draw[color=black] (5.74,-3.5) node {\large{$b_{1,0}$}};
\fill [color=black] (6.24,-0.32) circle (1.5pt);
\draw[color=black] (7.2,0.3) node {\large{$b_{1,\vert V(B_1)\vert-1}$}};
\fill [color=black] (6.58,1.76) circle (1.5pt);
\draw[color=black] (7.32,2.02) node {\large{$b_{4,0}$}};
\fill [color=black] (6.68,4.92) circle (1.5pt);
\draw[color=black] (8.46,5.3) node {\large{$b_{4,\vert V(B_4)\vert-1}$}};
\fill [color=black] (-0.08,-2.62) circle (1.5pt);
\draw[color=black] (0.66,-2.75) node {\large{$b_{7,0}$}};
\fill [color=black] (-3.72,1.02) circle (1.5pt);
\draw[color=black] (-4,1.5) node {\large{$b_{7,\vert V(B_7)\vert-1}$}};
\fill [color=black] (-3.32,3.16) circle (1.5pt);
\draw[color=black] (-4,3.42) node {\large{$b_{8,0}$}};
\fill [color=black] (1.68,-1.72) circle (1.5pt);
\draw[color=black] (1.85,-2.1) node {\large{$b_{9,0}$}};
\fill [color=black] (-1.96,1.92) circle (1.5pt);
\draw[color=black] (-2.6,2.25) node {\large{$b_{9,1}$}};
\fill [color=black] (3.02,4.34) circle (1.5pt);
\draw[color=black] (3.5,4.7) node {\large{$b_{9,2}$}};
\fill [color=black] (1.66,5.58) circle (1.5pt);
\draw[color=black] (2,6.1) node {\large{$b_{8,\vert V(B_8)\vert-1}$}};
\fill [color=black] (3.44,-1.4) circle (1.5pt);
\draw[color=black] (4,-2) node {\large{$b_{10,0}$}};
\fill [color=black] (4.68,1.5) circle (1.5pt);
\draw[color=black] (5.5,1.76) node {\large{$b_{10,1}$}};
\fill [color=black] (4.78,4.66) circle (1.5pt);
\draw[color=black] (5.35,5) node {\large{$b_{10,2}$}};
\draw[color=red] (5.98,-1.74) node {\Large{1}};
\draw[color=red] (7.02,3.2) node {\Large{4}};
\draw[color=red] (-1.6,-0.42) node {\Large{7}};
\draw[color=red] (-0.6,3.99) node {\Large{8}};
\draw[color=red] (0.16,0.48) node {\Large{7}};
\draw[color=red] (0.74,2.75) node {\Large{8}};
\draw[color=red] (4.42,0.08) node {\Large{1}};
\draw[color=red] (5.12,3.22) node {\Large{4}};
\draw[color=red] (3.75,1.72) node {\Large{9}};
\draw[color=red] (2.74,1.4) node {\Large{9}};
\end{scriptsize}
\end{tikzpicture}
	\caption{The graph $B$ in Case 1}
	\label{B6}
\end{figure}
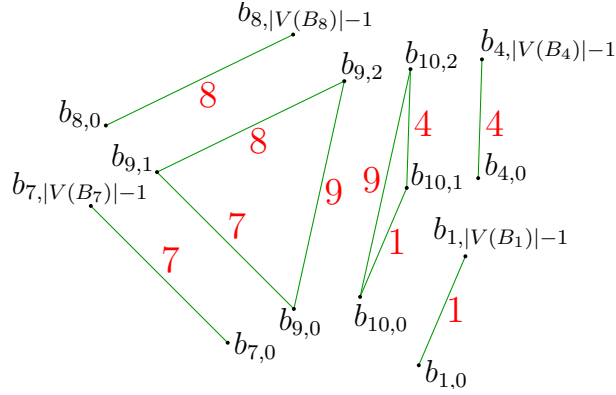

In case 2, the numbers of vertices of the $B_i$ are as follows:

$\left\lvert V(B_1)\right\rvert=\left\lvert V(B_4)\right\rvert=\sum\limits_{i=0}^{\frac{deg(q_p)}{2}}{\beta_{2i}+1}$.

$\left\lvert V(B_2)\right\rvert=\left\lvert V(B_5)\right\rvert=\sum\limits_{i=0}^{\frac{deg(q_p)}{2}}{\beta_{2i+1}+1}$.

$\left\lvert V(B_3)\right\rvert=\left\lvert V(B_6)\right\rvert=\left\lvert V(B_9)\right\rvert=\left\lvert V(B_{10})\right\rvert=3$.

$\left\lvert V(B_7)\right\rvert=2\cdot\sum\limits_{i=0}^{\frac{deg(q_n)}{2}}{\beta_{2i}+1}$.

$\left\lvert V(B_8)\right\rvert=2\cdot\sum\limits_{i=0}^{\frac{deg(q_n)}{2}}{\beta_{2i+1}+1}$.

While the green edges are located in the way drawn in Figure \ref{B7}.

\begin{figure}
	\centering
\pagestyle{empty}
\definecolor{qqzzqq}{rgb}{0,0.6,0}
\begin{tikzpicture}[line cap=round,line join=round,scale=0.5]
\clip(-5.5,-7.68) rectangle (17,11.5);
\draw [color=qqzzqq] (7.14,-6.18)-- (8.4,-3.84);
\draw [color=qqzzqq] (8.6,-2.36)-- (8.62,0.24);
\draw [color=qqzzqq] (8.46,2.02)-- (8.46,4.58);
\draw [color=qqzzqq] (8.06,5.86)-- (6.66,8.68);
\draw [color=qqzzqq] (5.16,-3.98)-- (6.42,-1.64);
\draw [color=qqzzqq] (6.42,-1.64)-- (6.44,0.96);
\draw [color=qqzzqq] (6.02,2.84)-- (6.02,5.4);
\draw [color=qqzzqq] (6.02,5.4)-- (4.62,8.22);
\draw [color=qqzzqq] (-0.08,-2.62)-- (-3.72,1.02);
\draw [color=qqzzqq] (-3.32,3.16)-- (0.2,9.84);
\draw [color=qqzzqq] (1.68,-1.72)-- (-1.96,1.92);
\draw [color=qqzzqq] (-1.96,1.92)-- (1.56,8.6);
\draw [color=qqzzqq] (3.44,-1.4)-- (4.72,3.54);
\draw [color=qqzzqq] (4.72,3.54)-- (3.32,8.92);
\draw [color=qqzzqq] (3.44,-1.4)-- (3.32,8.92);
\draw [color=qqzzqq] (5.16,-3.98)-- (6.44,0.96);
\draw [color=qqzzqq] (6.02,2.84)-- (4.62,8.22);
\draw [color=qqzzqq] (1.68,-1.72)-- (1.56,8.6);
\begin{scriptsize}
\fill [color=black] (7.14,-6.18) circle (1.5pt);
\draw[color=black] (7.88,-6.5) node {\large{$b_{1,0}$}};
\fill [color=black] (8.4,-3.84) circle (1.5pt);
\draw[color=black] (10.18,-3.4) node {\large{$b_{1,\vert V(B_1)\vert-1}$}};
\fill [color=black] (8.6,-2.36) circle (1.5pt);
\draw[color=black] (9.5,-2.3) node {\large{$b_{2,0}$}};
\fill [color=black] (8.62,0.24) circle (1.5pt);
\draw[color=black] (10.4,0.5) node {\large{$b_{2,\vert V(B_2)\vert-1}$}};
\fill [color=black] (8.46,2.02) circle (1.5pt);
\draw[color=black] (9.2,2.28) node {\large{$b_{4,0}$}};
\fill [color=black] (8.46,4.58) circle (1.5pt);
\draw[color=black] (10.24,4.84) node {\large{$b_{4,\vert V(B_4)\vert-1}$}};
\fill [color=black] (8.06,5.86) circle (1.5pt);
\draw[color=black] (8.8,6.12) node {\large{$b_{5,0}$}};
\fill [color=black] (6.66,8.68) circle (1.5pt);
\draw[color=black] (7.5,9.2) node {\large{$b_{5,\vert V(B_5)\vert-1}$}};
\fill [color=black] (-0.08,-2.62) circle (1.5pt);
\draw[color=black] (0.66,-2.36) node {\large{$b_{7,0}$}};
\fill [color=black] (-3.72,1.02) circle (1.5pt);
\draw[color=black] (-3.5,1.58) node {\large{$b_{7,\vert V(B_7)\vert-1}$}};
\fill [color=black] (-3.32,3.16) circle (1.5pt);
\draw[color=black] (-3.8,2.8) node {\large{$b_{8,0}$}};
\fill [color=black] (5.16,-3.98) circle (1.5pt);
\draw[color=black] (5.3,-4.4) node {\large{$b_{3,0}$}};
\fill [color=black] (6.42,-1.64) circle (1.5pt);
\draw[color=black] (7.16,-1.38) node {\large{$b_{3,1}$}};
\fill [color=black] (6.44,0.96) circle (1.5pt);
\draw[color=black] (7.18,1.22) node {\large{$b_{3,2}$}};
\fill [color=black] (6.02,2.84) circle (1.5pt);
\draw[color=black] (5.96,2.3) node {\large{$b_{6,0}$}};
\fill [color=black] (6.02,5.4) circle (1.5pt);
\draw[color=black] (6.76,5.66) node {\large{$b_{6,1}$}};
\fill [color=black] (4.62,8.22) circle (1.5pt);
\draw[color=black] (5,8.6) node {\large{$b_{6,2}$}};
\fill [color=black] (1.68,-1.72) circle (1.5pt);
\draw[color=black] (2.1,-2.4) node {\large{$b_{9,0}$}};
\fill [color=black] (-1.96,1.92) circle (1.5pt);
\draw[color=black] (-2.6,2.18) node {\large{$b_{9,1}$}};
\fill [color=black] (1.56,8.6) circle (1.5pt);
\draw[color=black] (1.45,9.15) node {\large{$b_{9,2}$}};
\fill [color=black] (0.2,9.84) circle (1.5pt);
\draw[color=black] (0.98,10.4) node {\large{$b_{8,\vert V(B_8)\vert-1}$}};
\fill [color=black] (3.44,-1.4) circle (1.5pt);
\draw[color=black] (3.51,-1.89) node {\large{$b_{10,0}$}};
\fill [color=black] (4.72,3.54) circle (1.5pt);
\draw[color=black] (5.01,3.2) node {\large{$b_{10,1}$}};
\fill [color=black] (3.32,8.92) circle (1.5pt);
\draw[color=black] (3.29,9.43) node {\large{$b_{10,2}$}};
\draw[color=red] (8.12,-5.02) node {\Large{1}};
\draw[color=red] (9,-1.3) node {\Large{2}};
\draw[color=red] (8.86,3.46) node {\Large{4}};
\draw[color=red] (7.72,7.56) node {\Large{5}};
\draw[color=red] (6.14,-2.82) node {\Large{1}};
\draw[color=red] (6.82,-0.2) node {\Large{2}};
\draw[color=red] (6.42,4.28) node {\Large{4}};
\draw[color=red] (5.68,7.1) node {\Large{5}};
\draw[color=red] (-1.6,-0.42) node {\Large{7}};
\draw[color=red] (-1.2,6.5) node {\Large{8}};
\draw[color=red] (0.16,0.48) node {\Large{7}};
\draw[color=red] (0.16,5.26) node {\Large{8}};
\draw[color=red] (4.46,1.14) node {\Large{3}};
\draw[color=red] (4.4,6.46) node {\Large{6}};
\draw[color=red] (3.76,3.92) node {\Large{9}};
\draw[color=red] (5.18,-1.44) node {\Large{3}};
\draw[color=red] (4.9,5.61) node {\Large{6}};
\draw[color=red] (2,3.6) node {\Large{9}};
\end{scriptsize}
\end{tikzpicture}
	\caption{The graph $B$ in Case 2}
	\label{B7}
\end{figure}
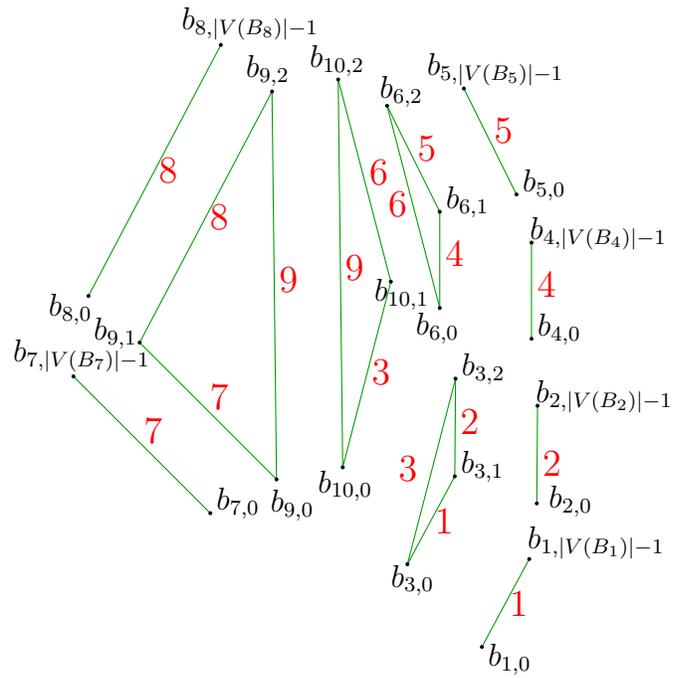

With the help of $B$ and three copies of $A$, we can define the desired $G\left(p\right)$:

Take three isomorphic copies of $A$: $A^{(1)}$, $A^{(2)}$ and $A^{(3)}$ and for any vertex $v$ of $A$, denote its respective copies in $A^{(1)}$, $A^{(2)}$ and $A^{(3)}$ by $v^{(1)}$, $v^{(2)}$ and $v^{(3)}$. Use similar notations for the complex numbers defined above: for example the vector that can be written as $1+\varepsilon+8\varepsilon^2+3\varepsilon^3$ in the coordinate system belonging to $A^{(1)}$ will be called $\left(1+\varepsilon+8\varepsilon^2+3\varepsilon^3\right)^{(1)}$.

Now for every $0\le l\le\beta_0-1$, connect the pair $\left(b_{1,l},b_{1,l+1}\right)$ and the pair $\left(a_{1,0}^{(1)},a_{1,1}^{(1)}\right)$ with a red grid and for every $1\le i\le\frac{deg(p)}{2}$ and every $\sum\limits_{j=0}^{i-1}{\beta_{2j}}\le l\le\sum\limits_{j=0}^i{\beta_{2j}}-1$, connect the pair $\left(b_{1,l},b_{1,l+1}\right)$ and the pair $\left(a_{2i,1}^{(1)},a_{2i,2}^{(1)}\right)$ with a red grid.

Similarly, if $V(B_2)\neq 0$ (so if Case 2 holds), then for all $0\le i\le\frac{deg(p)}{2}-1$ and $\sum\limits_{j=0}^{i-1}{\beta_{2j+1}}\le l\le\sum\limits_{j=0}^{i-1}{\beta_{2j+1}}$ connect the pair $\left(b_{2,l},b_{2,l+1}\right)$ and the pair $\left(a_{2i+1,1}^{(1)},a_{2i+1,2}^{(1)}\right)$ with a red grid.

Similarly, for every $0\le l\le\beta_0-1$, connect the pair $\left(b_{4,l},b_{4,l+1}\right)$ and the pair $\left(a_{1,0}^{(2)},a_{1,1}^{(2)}\right)$ with a red grid and for every $1\le i\le\frac{deg(p)}{2}$ and every $\sum\limits_{j=0}^{i-1}{\beta_{2j}}\le l\le\sum\limits_{j=0}^i{\beta_{2j}}-1$, connect the pair $\left(b_{4,l},b_{4,l+1}\right)$ and the pair $\left(a_{2i,1}^{(2)},a_{2i,2}^{(2)}\right)$ with a red grid.

Also, if $V(B_5)\neq 0$ (so if Case 2 holds), then for all $0\le i\le\frac{deg(p)}{2}-1$ and $\sum\limits_{j=0}^{i-1}{\beta_{2j+1}}\le l\sum\limits_{j=0}^{i-1}{\beta_{2j+1}}$ connect the pair $\left(b_{5,l},b_{5,l+1}\right)$ and the pair $\left(a_{2i+1,1}^{(2)},a_{2i+1,2}^{(2)}\right)$ with a red grid.

Also, for every $0\le l\le2\cdot\gamma_0-1$, connect the pair $\left(b_{7,l},b_{7,l+1}\right)$ and the pair $\left(a_{1,0}^{(3)},a_{1,1}^{(3)}\right)$ with a red grid and for every $1\le i\le\frac{deg(p)}{2}$ and every $2\cdot\sum\limits_{j=0}^{i-1}{\gamma_{2j}}\le l\le2\cdot\sum\limits_{j=0}^i{\gamma_{2j}}-1$, connect the pair $\left(b_{7,l},b_{7,l+1}\right)$ and the pair $\left(a_{2i,1}^{(3)},a_{2i,2}^{(3)}\right)$ with a red grid.

Finally, for all $0\le i\le\frac{deg(p)}{2}-1$ and $2\cdot\sum\limits_{j=0}^{i-1}{\gamma_{2j+1}}\le l\le2\cdot\sum\limits_{j=0}^{i-1}{\gamma_{2j+1}}-1$, connect the pair $\left(b_{8,l},b_{8,l+1}\right)$ and the pair $\left(a_{2i+1,1}^{(3)},a_{2i+1,2}^{(3)}\right)$ with a red grid.

Practically, these grid connections provided that the vectors $\vec{b_{1,i}b_{1,i+1}}$ are powers of $\varepsilon^{(1)}$ in an ascending order, summing up to $q_{p,e}^{(1)}$, and similarly, the vectors connecting neighbours in $B_2$, $B_4$, $B_5$, $B_7$ and $B_8$ also form paths of powers of $\varepsilon$ that sum up to $q_{p,o}^{(1)}$, $q_{p,e}^{(2)}$, $q_{p,o}^{(2)}$, $q_{n,o}^{(3)}$ and $q_{n,e}^{(3)}$, respectively.

This way, we constructed a PVEBG $\hat{G}\left(p\right)$, whose range is $S_0\left(p,0,\infty\right)$ and from the above description of its $(1,d)$-representations, it also has a number $W$, for which $Ran_W\left(\hat{G}\left(p\right)\right)=S_0\left(p,0,\infty\right)$ also applies, so by Proposition \ref{prop:pvebgtoebg}, we can create an EBG $G\left(p\right)$ for which $Ran\left(G\left(p\right)\right)=S_0\left(p,0,\infty\right)$.

\begin{lemma}\label{lem:egyikirany}
$G\left(p\right)$ does not have a $(1,d)$-representation if $d\notin S_0(p,0,+\infty)$.
\end{lemma}

\begin{proof}
First, we prove that $G\left(p\right)$ cannot have any $(1,d)$-representation for $d\notin S_0\left(p,0,+\infty\right)$.

Suppose that for some $d$, there exists a $(1,d)$-representation.

From Proposition \ref{prop:gridconnection}, we know that if Case 1 holds for $q(\varepsilon)$, then
\begin{align*}
\left\lvert\vec{b_{10,0}b_{10,1}}\right\rvert & =\left\lvert\vec{b_{1,0}b_{1,\left\lvert V(B_1)\right\rvert-1}}\right\rvert=\left\lvert\sum\limits_{i=0}^{\left\lvert V(B_1)\right\rvert-1}{2}{\left(\vec{b_{1,i}b_{1,i+1}}\right)}\right\rvert=\left\lvert\left(\sum\limits_{i=0}^{\frac{deg(p)}{2}}{\beta_{2i}\cdot\varepsilon^{2i}}\right)^{(1)}\right\rvert= \\
& =\left\lvert\left(q_p(\varepsilon)\right)^{(1)}\right\rvert=p_p(d).
\end{align*}
If Case 2 holds for $q(\varepsilon)$, then similarly
\begin{align*}
\left\lvert\vec{b_{10,0}b_{10,1}}\right\rvert & =\left\lvert\vec{b_{3,0}b_{3,2}}\right\rvert=\left\lvert\vec{b_{3,0}b_{3,1}}+\vec{b_{3,1}b_{3,2}}\right\rvert= \\
& =\left\lvert\left(\sum\limits_{i=0}^{\frac{deg(p)}{2}}{\beta_{2i}\cdot\varepsilon^{2i}}+\sum\limits_{i=0}^{\frac{deg(p)}{2}-1}{\beta_{2i+1}\cdot\varepsilon^{2i+1}}\right)^{(1)}\right\rvert=\left\lvert q_p(\varepsilon)\right\rvert=\left\lvert\left(q_p(\varepsilon)\right)^{(1)}\right\rvert=p_p(d).
\end{align*}
Similarly, we know that if Case 1 holds for $q(\varepsilon)$, then $\left\lvert\vec{b_{10,1}b_{10,2}}\right\rvert=\left\lvert\vec{b_{4,0}b_{4,\left\lvert V(B_4)\right\rvert-1}}\right\rvert=\left\lvert\left(\sum\limits_{i=0}^{\frac{deg(p)}{2}}{\beta_{2i}\cdot\varepsilon^{2i}}\right)^{(2)}\right\rvert=\left\lvert\left(q_p(\varepsilon)\right)^{(2)}\right\rvert=p_p(d)$. If Case 2 holds for $q(\varepsilon)$, then similarly $\left\lvert\vec{b_{10,0}b_{10,1}}\right\rvert=\left\lvert\vec{b_{6,0}b_{6,2}}\right\rvert=\left\lvert\vec{b_{6,0}b_{6,1}}+\vec{b_{6,1}b_{6,2}}\right\rvert=\left\lvert\left(\sum\limits_{i=0}^{\frac{deg(p)}{2}}{\beta_{2i}\cdot\varepsilon^{2i}}+\sum\limits_{i=0}^{\frac{deg(p)}{2}-1}{\beta_{2i+1}\cdot\varepsilon^{2i+1}}\right)^{(2)}\right\rvert=\left\lvert q_p(\varepsilon)\right\rvert=\left\lvert\left(q_p(\varepsilon)\right)^{(2)}\right\rvert=p_p(d)$.

And finally, we know that
\[
\left\lvert\vec{b_{10,2}b_{10,3}}\right\rvert=\left\lvert\vec{b_{7,0}b_{7,\left\lvert V(B_7)-1\right\rvert}}+\vec{b_{8,0}b_{8,\left\lvert V(B_8)-1\right\rvert}}\right\rvert=\left\lvert\left(q_{n,o}+q_{n,e}\right)^{(3)}\right\rvert=\left\lvert\left(2\cdot q_p(\varepsilon)\right)^{(3)}\right\rvert=2p_n(d).
\]

And if $d\notin S_0(p,0,+\infty)$, then $p_p(d)-p_n(d)=p(d)<0$, thus, $\left\lvert\vec{b_{10,0}b_{10,1}}\right\rvert+\left\lvert\vec{b_{10,1}b_{10,2}}\right\rvert=p_p(d)+p_p(d)<2p_n(d)=\left\lvert\vec{b_{10,0}b_{10,2}}\right\rvert$, which proves the statement.
\end{proof}

\begin{lemma}\label{lem:masikirany}
If $d\in S_0\left(p,0+\infty\right)$, then $G\left(p\right)$ does have a $(1,d)$-representation.
\end{lemma}

\begin{proof}
The only possible positions of the vertices of the $B_i$ ($1\le i\le 10$) were indicated above in the definition of $G\left(p\right)$ and in the proof of Lemma \ref{lem:egyikirany}.

Now we will first prove that if $d\in S_0(p,0,+\infty)$, then these are indeed $(1,d)$-representations for all the $B_i$ ($1\le i\le 10$) without any coincidences.

The drawing for $B_1$ is a $(1,d)$-representation: we only have to prove that there exists a vector $v$ so that $\left\langle v,\vec{b_{1,i},b_{1,i+1}}\right\rangle>0$ for $0\le i\le\left\lvert V(B_1)\right\rvert-2$, since in this $\left\langle v,\vec{b_{1,i},b_{1,j}}\right\rangle>0$ holds for any $1\le i<j\le\left\lvert V(B_1)\right\rvert-2$, which means that $\vec{b_{1,i},b_{1,j}}\neq 0$. And such a vector exists: 

All vectors $\vec{b_{1,i},b_{1,i+1}}$ ($0\le i\le\left\lvert V(B_1)-2\right\rvert$ are of the form $\left(\varepsilon^{2j}\right)^{(1)}$ $\left(0\le j\le\frac{deg(p)}{2}\right)$. So now using Lemma \ref{lem:argepsilon}, if we choose any $\varepsilon'$, for which $arg(\varepsilon')\in\left(arg\left(\varepsilon^{deg(p)}\right)-\frac{\pi}{2},\frac{\pi}{2}\right)$, then $\left(\varepsilon'\right)^{(1)}$ suffices. 

Similarly, the drawings for $B_2$, $B_4$, $B_5$, $B_7$ and $B_8$ also are $(1,d)$-representations.

The drawing for $B_3$ is also a $(1,d)$-representation:

If Case 1 holds, then $\left\lvert b_{3,0}b_{3,1}\right\rvert=p_p>0$, thus the statement is proved. If Cases 2 holds, then the equality $\left(q_{p,e}\right)^{(1)}+\left(q_{p,o}\right)^{(1)}=\left(q_p\right)^{(1)}$ holds, thus we only have to check if there are any coincidences. And since $\left\lvert b_{3,0}b_{3,1}\right\rvert=\left\lvert\left(q_{p,e}\right)^{(1)}\right\rvert>0$, $\left\lvert b_{3,1}b_{3,2}\right\rvert=\left\lvert\left(q_{p,o}\right)^{(1)}\right\rvert>0$ and $\left\lvert b_{3,0}b_{3,2}\right\rvert=\left\lvert\left(q_p\right)^{(1)}\right\rvert=p_p>0$, coincidences cannot occur.

For the drawing of $B_6$, the proof goes exactly the same and for the drawing of $B_9$, it also goes the same with the exception that we use $p_n$, $q_n$, $q_{n,e}$ and $q_{n,o}$ instead of $p_p$, $q_p$, $q_{p,e}$ and $q_{p,o}$.

We also know from Lemma \ref{lem:a2} that for any $d\in S_0(p,0,+\infty)$, $A$ has a $(1,d)$-representation.

Also, from the drawings it is trivial that both $A$ and all the $B_i$ ($1\le i\le 10$) have a bounded diameter and also that for all the grids, the two pairs of points connected by them have equal vectors between them. And from here, we can apply Lemma \ref{prop:gridconnection} and this finishes the proof.
\end{proof}

Now we are also done with the proof of Proposition \ref{prop:main}.

\hypertarget{pf:s1}{}\subsection{Proof of Proposition \ref{prop:s1}}

We define $G'\left(p\right)$ by adding one extra vertex $b_{10,3}$ to $G\left(p\right)$ and connect the pairs $\left(b_{10,0},b_{10,1}\right)$ and $\left(b_{10,3},b_{10,2}\right)$ with a red grid.

From Lemma \ref{lem:subgraph}, we know that the range of $G'\left(p\right)$ is a subset of $S_0(p,0,+\infty)$. Now suppose that for any $d$, the $(1,d)$-representation of $G\left(p\right)$ does not work here. This only can happen if $b_{10,3}$ would coincide with one of the vertices of $B_{10}$, otherwise we could substitute $B_{10}$ by the graph induced by the vertices of the grid between $\left(b_{10,0},b_{10,1}\right)$ and $\left(b_{10,3},b_{10,2}\right)$ (including the four vertices in its corners) and still get a good proof (it still remains to have a bounded diameter and the equalities between the vectors that are induced by the grids are still satisfied). And since the grid induces $\vec{b_{10,3}b_{10,2}}=\vec{b_{10,0}b_{10,1}}$ and $\vec{b_{10,3}b_{10,0}}=\vec{b_{10,2}b_{10,1}}$, the only vertex $b_{10,3}$ can coincide with is $b_{10,1}$. And this happens exactly if $\vec{b_{10,0}b_{10,3}}=\vec{b_{10,0}b_{10,1}}$ meaning that $\vec{b_{10,0}b_{10,2}}=\vec{b_{10,0}b_{10,3}}+\vec{b_{10,3}b_{10,2}}=2\vec{b_{10,0}b_{10,1}}$, thus $2p_n=2p_p\Rightarrow p=0$.\qed

\begin{proposition}\label{prop:valami}
For an EBG $G$ and a positive rational number $L=\frac{L_1}{L_2}$ ($L_1$ and $L_2$ are relative primes and are in $\mathbb{Z}_{>0}$) and an arbitrary real number $U\ge L$, if $ran(G)\cap\left[L,+\infty\right)\neq\emptyset$, then one can construct a graph $G_L^U$ for which $ran\left(G_L^U\right)\supseteq\left(0,L\right]\cup ran(G)\setminus\left(U,+\infty\right)$ and $ran\left(G_L^U\right)\cap\left(L,+\infty\right)\setminus ran(G)=\emptyset$.
\end{proposition}

\begin{proof}
Define graphs $C$ and $D$ in the following way, where $r=\left\lceil\frac{U}{2L}\right\rceil+1$:

\begin{figure}
\begin{minipage}{.49\textwidth}
	\centering
	\usetikzlibrary{arrows}
\definecolor{qqzzqq}{rgb}{0,0.6,0}
\definecolor{ffqqqq}{rgb}{1,0,0}
\begin{tikzpicture}[line cap=round,line join=round,scale=0.9]
\clip(-2.77,-1.44) rectangle (5.58,5.7);
\draw [color=ffqqqq] (1,0)-- (0,0);
\draw [color=ffqqqq] (0,0)-- (0.5,0.87);
\draw [color=ffqqqq] (0.5,0.87)-- (1,0);
\draw [color=ffqqqq] (0.5,0.87)-- (-0.5,0.87);
\draw [color=ffqqqq] (-0.5,0.87)-- (0,0);
\draw [shift={(0.5,0.87)},color=qqzzqq]  plot[domain=4.19:5.24,variable=\t]({1*1*cos(\t r)+0*1*sin(\t r)},{0*1*cos(\t r)+1*1*sin(\t r)});
\draw [shift={(0.5,0.87)},color=qqzzqq]  plot[domain=3.14:4.19,variable=\t]({1*1*cos(\t r)+0*1*sin(\t r)},{0*1*cos(\t r)+1*1*sin(\t r)});
\draw [shift={(-1.5,4.33)},color=qqzzqq]  plot[domain=3.14:4.19,variable=\t]({1*1*cos(\t r)+0*1*sin(\t r)},{0*1*cos(\t r)+1*1*sin(\t r)});
\draw [shift={(-1,3.46)},dash pattern=on 2pt off 2pt,color=qqzzqq]  plot[domain=3.14:4.19,variable=\t]({1*1*cos(\t r)+0*1*sin(\t r)},{0*1*cos(\t r)+1*1*sin(\t r)});
\draw [shift={(-0.5,2.6)},color=qqzzqq]  plot[domain=3.14:4.19,variable=\t]({1*1*cos(\t r)+0*1*sin(\t r)},{0*1*cos(\t r)+1*1*sin(\t r)});
\draw [shift={(0,1.73)},color=qqzzqq]  plot[domain=3.14:4.19,variable=\t]({1*1*cos(\t r)+0*1*sin(\t r)},{0*1*cos(\t r)+1*1*sin(\t r)});
\draw [shift={(1.5,0.87)},color=qqzzqq]  plot[domain=4.19:5.24,variable=\t]({1*1*cos(\t r)+0*1*sin(\t r)},{0*1*cos(\t r)+1*1*sin(\t r)});
\draw [shift={(2.5,0.87)},color=qqzzqq]  plot[domain=4.19:5.24,variable=\t]({1*1*cos(\t r)+0*1*sin(\t r)},{0*1*cos(\t r)+1*1*sin(\t r)});
\draw [shift={(3.5,0.87)},dash pattern=on 2pt off 2pt,color=qqzzqq]  plot[domain=4.19:5.24,variable=\t]({1*1*cos(\t r)+0*1*sin(\t r)},{0*1*cos(\t r)+1*1*sin(\t r)});
\draw [shift={(4.5,0.87)},color=qqzzqq]  plot[domain=4.19:5.24,variable=\t]({1*1*cos(\t r)+0*1*sin(\t r)},{0*1*cos(\t r)+1*1*sin(\t r)});
\draw [shift={(1,0)},color=qqzzqq]  plot[domain=4.19:5.24,variable=\t]({0.5*1*cos(\t r)+0.87*1*sin(\t r)},{0.87*1*cos(\t r)+-0.5*1*sin(\t r)});
\draw [shift={(1.5,0.87)},color=qqzzqq]  plot[domain=4.19:5.24,variable=\t]({0.5*1*cos(\t r)+0.87*1*sin(\t r)},{0.87*1*cos(\t r)+-0.5*1*sin(\t r)});
\draw [shift={(2,1.73)},dash pattern=on 2pt off 2pt,color=qqzzqq]  plot[domain=4.19:5.24,variable=\t]({0.5*1*cos(\t r)+0.87*1*sin(\t r)},{0.87*1*cos(\t r)+-0.5*1*sin(\t r)});
\draw [shift={(2.5,2.6)},dash pattern=on 2pt off 2pt,color=qqzzqq]  plot[domain=4.19:5.24,variable=\t]({0.5*1*cos(\t r)+0.87*1*sin(\t r)},{0.87*1*cos(\t r)+-0.5*1*sin(\t r)});
\draw [shift={(3,3.46)},color=qqzzqq]  plot[domain=4.19:5.24,variable=\t]({0.5*1*cos(\t r)+0.87*1*sin(\t r)},{0.87*1*cos(\t r)+-0.5*1*sin(\t r)});
\draw [shift={(3.5,4.33)},color=qqzzqq]  plot[domain=4.19:5.24,variable=\t]({0.5*1*cos(\t r)+0.87*1*sin(\t r)},{0.87*1*cos(\t r)+-0.5*1*sin(\t r)});
\draw [shift={(0.5,-0.87)},color=qqzzqq]  plot[domain=4.19:5.24,variable=\t]({0.5*1*cos(\t r)+0.87*1*sin(\t r)},{0.87*1*cos(\t r)+-0.5*1*sin(\t r)});
\draw [shift={(0,0)},color=qqzzqq]  plot[domain=4.19:5.24,variable=\t]({1*1*cos(\t r)+0*1*sin(\t r)},{0*1*cos(\t r)+1*1*sin(\t r)});
\draw [shift={(2.09,3.01)},color=qqzzqq]  plot[domain=4.32:5.48,variable=\t]({1*4.18*cos(\t r)+0*4.18*sin(\t r)},{0*4.18*cos(\t r)+1*4.18*sin(\t r)});
\draw [color=qqzzqq] (5,0)-- (3.5,4.33);
\draw [shift={(4,5.2)},color=qqzzqq]  plot[domain=4.19:5.24,variable=\t]({-0.5*1*cos(\t r)+0.87*1*sin(\t r)},{0.87*1*cos(\t r)+0.5*1*sin(\t r)});
\draw [color=qqzzqq] (-2.5,4.33)-- (3,5.2);
\begin{scriptsize}
\fill [color=black] (0,0) circle (1.5pt);
\draw[color=black] (0.34,0.12) node {\large{$c_{1,0}$}};
\fill [color=black] (1,0) circle (1.5pt);
\draw[color=black] (1.34,0.12) node {\large{$c_{1,1}$}};
\fill [color=black] (2,0) circle (1.5pt);
\draw[color=black] (2.34,0.12) node {\large{$c_{1,2}$}};
\fill [color=black] (3,0) circle (1.5pt);
\draw[color=black] (3.34,0.12) node {\large{$c_{1,3}$}};
\fill [color=black] (4,0) circle (1.5pt);
\draw[color=black] (4.42,0.12) node {\large{$c_{1,2r-1}$}};
\fill [color=black] (5,0) circle (1.5pt);
\draw[color=black] (5.36,0.12) node {\large{$c_{1,2r}$}};
\fill [color=black] (-0.5,0.87) circle (1.5pt);
\draw[color=black] (-0.15,0.98) node {\large{$c_{1,-1}$}};
\fill [color=black] (-1,1.73) circle (1.5pt);
\draw[color=black] (-0.65,1.85) node {\large{$c_{1,-2}$}};
\fill [color=black] (-1.5,2.6) circle (1.5pt);
\draw[color=black] (-1.15,2.72) node {\large{$c_{1,-3}$}};
\fill [color=black] (-2,3.46) circle (1.5pt);
\draw[color=black] (-1.54,3.58) node {\large{$c_{1,-2r+1}$}};
\fill [color=black] (-2.5,4.33) circle (1.5pt);
\draw[color=black] (-2.12,4.45) node {\large{$c_{1,-2r}$}};
\fill [color=black] (0.5,0.87) circle (1.5pt);
\draw[color=black] (0.91,0.98) node {\large{$c_{1,1^*}$}};
\draw[color=red] (0.59,-0.36) node {\large{1}};
\draw[color=red] (-0.28,0.49) node {\large{3}};
\draw[color=red] (-2.28,3.95) node {\large{3}};
\draw[color=red] (-1.28,2.22) node {\large{3}};
\draw[color=red] (-0.78,1.36) node {\large{3}};
\draw[color=red] (1.59,-0.36) node {\large{1}};
\draw[color=red] (2.59,-0.36) node {\large{1}};
\draw[color=red] (4.59,-0.36) node {\large{1}};
\draw[color=red] (0.22,0.62) node {\large{2}};
\draw[color=red] (0.73,1.49) node {\large{2}};
\draw[color=red] (2.23,4.09) node {\large{2}};
\fill [color=black] (1,1.73) circle (1.5pt);
\draw[color=black] (1.41,1.85) node {\large{$c_{1,2^*}$}};
\fill [color=black] (1.5,2.6) circle (1.5pt);
\draw[color=black] (1.9,2.72) node {\large{$c_{1,r^*}$}};
\fill [color=black] (2,3.46) circle (1.5pt);
\draw[color=black] (2.66,3.54) node {\large{$c_{1,2r-1^*}$}};
\fill [color=black] (2.5,4.33) circle (1.5pt);
\draw[color=black] (2.94,4.45) node {\large{$c_{1,2r^*}$}};
\draw[color=red] (2.73,4.95) node {\large{2}};
\fill [color=black] (3,5.2) circle (1.5pt);
\draw[color=black] (3.52,5.31) node {\large{$c_{1,2r+1^*}$}};
\draw[color=red] (-0.28,-0.24) node {\large{2}};
\fill [color=black] (-0.5,-0.87) circle (1.5pt);
\draw[color=black] (-0.08,-0.75) node {\large{$c_{1,-1^*}$}};
\fill [color=black] (3.5,4.33) circle (1.5pt);
\draw[color=black] (3.98,4.45) node {\large{$c_{1,2r+2*}$}};
\fill [color=black] (0.5,-0.87) circle (1.5pt);
\draw[color=black] (0.93,-0.75) node {\large{$c_{1,-2^*}$}};
\draw[color=red] (0.09,-1.23) node {\large{1}};
\draw[color=red] (2.97,-0.98) node {\large{4}};
\draw[color=red] (4.42,2.29) node {\large{5}};
\draw[color=red] (3.23,4.82) node {\large{3}};
\draw[color=red] (0.31,4.68) node {\large{6}};
\end{scriptsize}
\end{tikzpicture}
	\captionof{figure}{$C_1$ drawn in its only possible $(1,d)$-representation up to isometry}
	\label{C6}
\end{minipage}
\begin{minipage}{.49\textwidth}
	\centering
	\usetikzlibrary{arrows}
\definecolor{qqzzqq}{rgb}{0,0.6,0}
\definecolor{ffqqqq}{rgb}{1,0,0}
\begin{tikzpicture}[line cap=round,line join=round,scale=0.9]
\clip(-2.98,-2.88) rectangle (5.75,5.82);
\draw [color=ffqqqq] (1,0)-- (0,0);
\draw [color=ffqqqq] (0,0)-- (0.5,0.87);
\draw [color=ffqqqq] (-0.5,0.87)-- (0,0);
\draw [shift={(0.5,0.87)},color=qqzzqq]  plot[domain=4.19:5.24,variable=\t]({1*1*cos(\t r)+0*1*sin(\t r)},{0*1*cos(\t r)+1*1*sin(\t r)});
\draw [shift={(0.5,0.87)},color=qqzzqq]  plot[domain=3.14:4.19,variable=\t]({1*1*cos(\t r)+0*1*sin(\t r)},{0*1*cos(\t r)+1*1*sin(\t r)});
\draw [shift={(-1.5,4.33)},color=qqzzqq]  plot[domain=3.14:4.19,variable=\t]({1*1*cos(\t r)+0*1*sin(\t r)},{0*1*cos(\t r)+1*1*sin(\t r)});
\draw [shift={(-1,3.46)},dash pattern=on 2pt off 2pt,color=qqzzqq]  plot[domain=3.14:4.19,variable=\t]({1*1*cos(\t r)+0*1*sin(\t r)},{0*1*cos(\t r)+1*1*sin(\t r)});
\draw [shift={(-0.5,2.6)},color=qqzzqq]  plot[domain=3.14:4.19,variable=\t]({1*1*cos(\t r)+0*1*sin(\t r)},{0*1*cos(\t r)+1*1*sin(\t r)});
\draw [shift={(0,1.73)},color=qqzzqq]  plot[domain=3.14:4.19,variable=\t]({1*1*cos(\t r)+0*1*sin(\t r)},{0*1*cos(\t r)+1*1*sin(\t r)});
\draw [shift={(1.5,0.87)},color=qqzzqq]  plot[domain=4.19:5.24,variable=\t]({1*1*cos(\t r)+0*1*sin(\t r)},{0*1*cos(\t r)+1*1*sin(\t r)});
\draw [shift={(2.5,0.87)},color=qqzzqq]  plot[domain=4.19:5.24,variable=\t]({1*1*cos(\t r)+0*1*sin(\t r)},{0*1*cos(\t r)+1*1*sin(\t r)});
\draw [shift={(3.5,0.87)},dash pattern=on 2pt off 2pt,color=qqzzqq]  plot[domain=4.19:5.24,variable=\t]({1*1*cos(\t r)+0*1*sin(\t r)},{0*1*cos(\t r)+1*1*sin(\t r)});
\draw [shift={(4.5,0.87)},color=qqzzqq]  plot[domain=4.19:5.24,variable=\t]({1*1*cos(\t r)+0*1*sin(\t r)},{0*1*cos(\t r)+1*1*sin(\t r)});
\draw [color=ffqqqq] (0.5,0.87)-- (1,1.73);
\draw [dash pattern=on 2pt off 2pt,color=ffqqqq] (1,1.73)-- (1.5,2.6);
\draw [dash pattern=on 2pt off 2pt,color=ffqqqq] (1.5,2.6)-- (2,3.46);
\draw [color=ffqqqq] (2,3.46)-- (2.5,4.33);
\draw [color=ffqqqq] (2.5,4.33)-- (3,5.2);
\draw [color=ffqqqq] (0,0)-- (-0.95,0.32);
\draw [color=ffqqqq] (-0.95,0.32)-- (-1.86,-0.1);
\draw [dash pattern=on 2pt off 2pt,color=ffqqqq] (-1.86,-0.1)-- (-1.51,-1.04);
\draw [dash pattern=on 2pt off 2pt,color=ffqqqq] (-1.51,-1.04)-- (-1.28,-2.02);
\draw [color=ffqqqq] (-1.28,-2.02)-- (-0.29,-1.84);
\draw [color=ffqqqq] (-0.29,-1.84)-- (-0.5,-0.87);
\draw [color=ffqqqq] (1,0)-- (1.5,-0.87);
\draw [color=ffqqqq] (1.5,-0.87)-- (1.45,-1.87);
\draw [dash pattern=on 2pt off 2pt,color=ffqqqq] (1.45,-1.87)-- (0.78,-2.62);
\draw [dash pattern=on 2pt off 2pt,color=ffqqqq] (0.78,-2.62)-- (0.18,-1.82);
\draw [color=ffqqqq] (0.18,-1.82)-- (0.5,-0.87);
\draw [color=ffqqqq] (1,0)-- (1.5,0.87);
\draw [color=ffqqqq] (1.5,0.87)-- (2,1.73);
\draw [dash pattern=on 2pt off 2pt,color=ffqqqq] (2,1.73)-- (2.5,2.6);
\draw [dash pattern=on 2pt off 2pt,color=ffqqqq] (2.5,2.6)-- (3,3.46);
\draw [color=ffqqqq] (3,3.46)-- (3.5,4.33);
\draw [shift={(0,0)},color=qqzzqq]  plot[domain=4.19:5.24,variable=\t]({1*1*cos(\t r)+0*1*sin(\t r)},{0*1*cos(\t r)+1*1*sin(\t r)});
\draw [shift={(2.09,3.01)},color=qqzzqq]  plot[domain=4.32:5.48,variable=\t]({1*4.18*cos(\t r)+0*4.18*sin(\t r)},{0*4.18*cos(\t r)+1*4.18*sin(\t r)});
\draw [color=qqzzqq] (5,0)-- (3.5,4.33);
\draw [shift={(4,5.2)},color=qqzzqq]  plot[domain=4.19:5.24,variable=\t]({-0.5*1*cos(\t r)+0.87*1*sin(\t r)},{0.87*1*cos(\t r)+0.5*1*sin(\t r)});
\draw [color=qqzzqq] (-2.5,4.33)-- (3,5.2);
\draw [->] (1.5,2.6) -- (2.5,2.6);
\draw [->] (-1.51,-1.04) -- (0.78,-2.62);
\begin{scriptsize}
\fill [color=black] (0,0) circle (1.5pt);
\draw[color=black] (0.34,0.12) node {$c_{2,0}$};
\fill [color=black] (1,0) circle (1.5pt);
\draw[color=black] (1.34,0.12) node {$c_{2,1}$};
\fill [color=black] (2,0) circle (1.5pt);
\draw[color=black] (2.34,0.12) node {$c_{2,2}$};
\fill [color=black] (3,0) circle (1.5pt);
\draw[color=black] (3.34,0.12) node {$c_{2,3}$};
\fill [color=black] (4,0) circle (1.5pt);
\draw[color=black] (4.42,0.12) node {$c_{2,2r-1}$};
\fill [color=black] (5,0) circle (1.5pt);
\draw[color=black] (5.36,0.12) node {$c_{2,2r}$};
\fill [color=black] (-0.5,0.87) circle (1.5pt);
\draw[color=black] (-0.15,0.98) node {$c_{2,-1}$};
\fill [color=black] (-1,1.73) circle (1.5pt);
\draw[color=black] (-0.65,1.85) node {$c_{2,-2}$};
\fill [color=black] (-1.5,2.6) circle (1.5pt);
\draw[color=black] (-1.15,2.72) node {$c_{2,-3}$};
\fill [color=black] (-2,3.46) circle (1.5pt);
\draw[color=black] (-1.54,3.58) node {$c_{2,-2r+1}$};
\fill [color=black] (-2.5,4.33) circle (1.5pt);
\draw[color=black] (-2.12,4.45) node {$c_{2,-2r}$};
\fill [color=black] (0.5,0.87) circle (1.5pt);
\draw[color=black] (0.91,0.98) node {$c_{2,1^*}$};
\draw[color=red] (0.59,-0.01) node {8};
\draw[color=red] (-0.28,0.49) node {7};
\draw[color=red] (-2.28,3.95) node {7};
\draw[color=red] (-1.28,2.22) node {7};
\draw[color=red] (-0.78,1.36) node {7};
\draw[color=red] (1.59,-0.01) node {8};
\draw[color=red] (2.59,-0.01) node {8};
\draw[color=red] (4.59,-0.01) node {8};
\fill [color=black] (1,1.73) circle (1.5pt);
\draw[color=black] (1.41,1.85) node {$c_{2,2^*}$};
\fill [color=black] (1.5,2.6) circle (1.5pt);
\draw[color=black] (1.9,2.72) node {$c_{2,r^*}$};
\fill [color=black] (2,3.46) circle (1.5pt);
\draw[color=black] (2.66,3.54) node {$c_{2,2r-1^*}$};
\fill [color=black] (2.5,4.33) circle (1.5pt);
\draw[color=black] (2.94,4.45) node {$c_{2,2r^*}$};
\fill [color=black] (3,5.2) circle (1.5pt);
\draw[color=black] (3.52,5.31) node {$c_{2,2r+1^*}$};
\fill [color=black] (-0.5,-0.87) circle (1.5pt);
\draw[color=black] (0,-0.75) node {$c_{2,-2r-1^*}$};
\fill [color=black] (3.5,4.33) circle (1.5pt);
\draw[color=black] (3.87,4.45) node {$c_{2,2r'}$};
\fill [color=black] (0.5,-0.87) circle (1.5pt);
\draw[color=black] (0.89,-0.75) node {$c_{2,-2r'}$};
\fill [color=black] (-0.95,0.32) circle (1.5pt);
\draw[color=black] (-0.52,0.44) node {$c_{2,-1^*}$};
\fill [color=black] (-1.86,-0.1) circle (1.5pt);
\draw[color=black] (-1.43,0.02) node {$c_{2,-2^*}$};
\fill [color=black] (-1.51,-1.04) circle (1.5pt);
\draw[color=black] (-1.1,-0.92) node {$c_{2,-r^*}$};
\fill [color=black] (-1.28,-2.02) circle (1.5pt);
\draw[color=black] (-0.75,-1.9) node {$c_{2,-2r+1^*}$};
\fill [color=black] (-0.29,-1.84) circle (1.5pt);
\draw[color=black] (0.16,-1.72) node {$c_{2,-2r^*}$};
\fill [color=black] (1.5,-0.87) circle (1.5pt);
\draw[color=black] (1.86,-0.75) node {$c_{2,-1'}$};
\fill [color=black] (1.45,-1.87) circle (1.5pt);
\draw[color=black] (1.82,-1.75) node {$c_{2,-2'}$};
\fill [color=black] (0.78,-2.62) circle (1.5pt);
\draw[color=black] (1.13,-2.5) node {$c_{2,-r'}$};
\fill [color=black] (0.18,-1.82) circle (1.5pt);
\draw[color=black] (0.65,-1.69) node {$c_{2,-2r+1'}$};
\fill [color=black] (1.5,0.87) circle (1.5pt);
\draw[color=black] (1.86,0.98) node {$c_{2,1'}$};
\fill [color=black] (2,1.73) circle (1.5pt);
\draw[color=black] (2.36,1.85) node {$c_{2,2'}$};
\fill [color=black] (2.5,2.6) circle (1.5pt);
\draw[color=black] (2.84,2.72) node {$c_{2,r'}$};
\fill [color=black] (3,3.46) circle (1.5pt);
\draw[color=black] (3.43,3.58) node {$c_{2,2r-1'}$};
\draw[color=red] (0.09,-0.88) node {1};
\draw[color=red] (2.97,-0.98) node {4};
\draw[color=red] (4.42,2.29) node {5};
\draw[color=red] (3.23,4.82) node {3};
\draw[color=red] (0.31,4.68) node {6};
\draw[color=black] (2.39,2.7) node {$\vartheta_1$};
\draw[color=black] (0.05,-1.73) node {$\vartheta_2$};
\fill [color=black] (0.25,-0.43) circle (1.5pt);
\draw[color=black] (0.47,-0.31) node {$O_1$};
\end{scriptsize}
\end{tikzpicture}
	\captionof{figure}{$C_2$ drawn in one of its possible $(1,d)$-representations: when $C$ is turned on}
	\label{C7}
\end{minipage}
\end{figure}
\begin{figure}
\begin{minipage}{.49\textwidth}
	\centering
	\usetikzlibrary{arrows}
\definecolor{uququq}{rgb}{0.25,0.25,0.25}
\definecolor{qqzzqq}{rgb}{0,0.6,0}
\definecolor{ffqqqq}{rgb}{1,0,0}
\begin{tikzpicture}[line cap=round,line join=round,scale=0.9]
\clip(-2.73,-1.38) rectangle (5.75,6.71);
\draw [color=ffqqqq] (0,0)-- (0.5,0.87);
\draw [shift={(3.5,4.33)},color=qqzzqq]  plot[domain=4.19:5.24,variable=\t]({0.5*1*cos(\t r)+0.87*1*sin(\t r)},{0.87*1*cos(\t r)+-0.5*1*sin(\t r)});
\draw [dash pattern=on 2pt off 2pt,color=ffqqqq] (0.5,0.87)-- (1,1.73);
\draw [dash pattern=on 2pt off 2pt,color=ffqqqq] (1,1.73)-- (1.5,2.6);
\draw [color=ffqqqq] (1.5,2.6)-- (2,3.46);
\draw [color=ffqqqq] (2,3.46)-- (2.5,4.33);
\draw [dash pattern=on 2pt off 2pt,color=ffqqqq] (1,0)-- (1.5,0.87);
\draw [dash pattern=on 2pt off 2pt,color=ffqqqq] (1.5,0.87)-- (2,1.73);
\draw [color=ffqqqq] (2,1.73)-- (2.5,2.6);
\draw [color=ffqqqq] (2.5,2.6)-- (3,3.46);
\draw [shift={(0,0)},color=qqzzqq]  plot[domain=4.19:5.24,variable=\t]({1*1*cos(\t r)+0*1*sin(\t r)},{0*1*cos(\t r)+1*1*sin(\t r)});
\draw [shift={(2.09,3.01)},color=qqzzqq]  plot[domain=4.32:5.48,variable=\t]({1*4.18*cos(\t r)+0*4.18*sin(\t r)},{0*4.18*cos(\t r)+1*4.18*sin(\t r)});
\draw [color=qqzzqq] (5,0)-- (3.5,4.33);
\draw [shift={(4,5.2)},color=qqzzqq]  plot[domain=4.19:5.24,variable=\t]({-0.5*1*cos(\t r)+0.87*1*sin(\t r)},{0.87*1*cos(\t r)+0.5*1*sin(\t r)});
\draw [color=qqzzqq] (-2.5,4.33)-- (3,5.2);
\draw [shift={(5.5,0.87)},color=qqzzqq]  plot[domain=4.19:5.24,variable=\t]({0.5*1*cos(\t r)+0.87*1*sin(\t r)},{-0.87*1*cos(\t r)+0.5*1*sin(\t r)});
\draw [shift={(5,1.73)},dash pattern=on 2pt off 2pt,color=qqzzqq]  plot[domain=4.19:5.24,variable=\t]({0.5*1*cos(\t r)+0.87*1*sin(\t r)},{-0.87*1*cos(\t r)+0.5*1*sin(\t r)});
\draw [shift={(4.5,2.6)},color=qqzzqq]  plot[domain=4.19:5.24,variable=\t]({0.5*1*cos(\t r)+0.87*1*sin(\t r)},{-0.87*1*cos(\t r)+0.5*1*sin(\t r)});
\draw [shift={(4,3.46)},color=qqzzqq]  plot[domain=4.19:5.24,variable=\t]({0.5*1*cos(\t r)+0.87*1*sin(\t r)},{-0.87*1*cos(\t r)+0.5*1*sin(\t r)});
\draw [shift={(3.5,4.33)},color=qqzzqq]  plot[domain=4.19:5.24,variable=\t]({0.5*1*cos(\t r)+0.87*1*sin(\t r)},{-0.87*1*cos(\t r)+0.5*1*sin(\t r)});
\draw [shift={(-2,5.2)},color=qqzzqq]  plot[domain=3.14:4.19,variable=\t]({0.5*1*cos(\t r)+-0.87*1*sin(\t r)},{0.87*1*cos(\t r)+0.5*1*sin(\t r)});
\draw [shift={(-1,5.2)},dash pattern=on 2pt off 2pt,color=qqzzqq]  plot[domain=3.14:4.19,variable=\t]({0.5*1*cos(\t r)+-0.87*1*sin(\t r)},{0.87*1*cos(\t r)+0.5*1*sin(\t r)});
\draw [shift={(0,5.2)},color=qqzzqq]  plot[domain=3.14:4.19,variable=\t]({0.5*1*cos(\t r)+-0.87*1*sin(\t r)},{0.87*1*cos(\t r)+0.5*1*sin(\t r)});
\draw [shift={(1,5.2)},color=qqzzqq]  plot[domain=3.14:4.19,variable=\t]({0.5*1*cos(\t r)+-0.87*1*sin(\t r)},{0.87*1*cos(\t r)+0.5*1*sin(\t r)});
\draw [shift={(2,5.2)},color=qqzzqq]  plot[domain=3.14:4.19,variable=\t]({0.5*1*cos(\t r)+-0.87*1*sin(\t r)},{0.87*1*cos(\t r)+0.5*1*sin(\t r)});
\draw [color=ffqqqq] (0,0)-- (-0.5,-0.87);
\draw [color=ffqqqq] (1,0)-- (0.5,-0.87);
\draw [color=ffqqqq] (2.5,4.33)-- (2.37,5.32);
\draw [color=ffqqqq] (2.37,5.32)-- (1.38,5.47);
\draw [dash pattern=on 2pt off 2pt,color=ffqqqq] (1.38,5.47)-- (2.05,6.21);
\draw [dash pattern=on 2pt off 2pt,color=ffqqqq] (2.05,6.21)-- (3.04,6.36);
\draw [color=ffqqqq] (3,3.46)-- (3.92,3.85);
\draw [color=ffqqqq] (3.92,3.85)-- (4.47,4.69);
\draw [dash pattern=on 2pt off 2pt,color=ffqqqq] (4.47,4.69)-- (4.36,5.68);
\draw [dash pattern=on 2pt off 2pt,color=ffqqqq] (4.36,5.68)-- (3.42,5.33);
\draw [color=ffqqqq] (3.42,5.33)-- (3.5,4.33);
\draw [color=ffqqqq] (3,5.2)-- (2.21,5.81);
\draw [color=ffqqqq] (2.21,5.81)-- (3.04,6.36);
\draw [->] (2.05,6.21) -- (4.36,5.68);
\draw [->] (1,1.73) -- (1.5,0.87);
\begin{scriptsize}
\fill [color=black] (0,0) circle (1.5pt);
\draw[color=black] (0.45,0.13) node {$c_{2,2r^*}$};
\fill [color=black] (1,0) circle (1.5pt);
\draw[color=black] (1.38,0.13) node {$c_{2,-2r-1'}$};
\fill [color=black] (5,0) circle (1.5pt);
\draw[color=black] (5.37,0.13) node {$c_{2,2r}$};
\fill [color=black] (-2.5,4.33) circle (1.5pt);
\draw[color=black] (-2.11,4.46) node {$c_{2,-2r}$};
\fill [color=black] (0.5,0.87) circle (1.5pt);
\draw[color=black] (1.05,0.99) node {$c_{2,-2r+1^*}$};
\fill [color=black] (1,1.73) circle (1.5pt);
\draw[color=black] (1.42,1.86) node {$c_{2,-r^*}$};
\fill [color=black] (1.5,2.6) circle (1.5pt);
\draw[color=black] (1.93,2.72) node {$c_{2,-2^*}$};
\fill [color=black] (2,3.46) circle (1.5pt);
\draw[color=black] (2.6,3.54) node {$c_{2,-1^*}$};
\fill [color=black] (2.5,4.33) circle (1.5pt);
\draw[color=black] (2.84,4.46) node {$c_{2,0}$};
\draw[color=qqzzqq] (2.72,4.95) node {2};
\fill [color=black] (3,5.2) circle (1.5pt);
\draw[color=black] (3.52,5.32) node {$c_{2,2r+1^*}$};
\fill [color=black] (-0.5,-0.87) circle (1.5pt);
\draw[color=black] (0.02,-0.74) node {$c_{2,-2r-1^*}$};
\fill [color=black] (3.5,4.33) circle (1.5pt);
\draw[color=black] (3.88,4.46) node {$c_{2,2r'}$};
\fill [color=black] (0.5,-0.87) circle (1.5pt);
\draw[color=black] (0.9,-0.74) node {$c_{2,-2r'}$};
\fill [color=black] (1.5,0.87) circle (1.5pt);
\draw[color=black] (1.84,0.99) node {$c_{2,r'}$};
\fill [color=black] (2,1.73) circle (1.5pt);
\draw[color=black] (2.36,1.86) node {$c_{2,2'}$};
\fill [color=black] (2.5,2.6) circle (1.5pt);
\draw[color=black] (2.85,2.72) node {$c_{2,1'}$};
\fill [color=black] (3,3.46) circle (1.5pt);
\draw[color=black] (3.33,3.59) node {$c_{2,1}$};
\draw[color=red] (0.09,-0.87) node {1};
\draw[color=red] (2.97,-0.97) node {4};
\draw[color=red] (4.43,2.28) node {5};
\draw[color=red] (3.22,4.82) node {3};
\draw[color=red] (0.3,4.68) node {6};
\draw[color=red] (4.73,0.49) node {8};
\fill [color=black] (4.5,0.87) circle (1.5pt);
\draw[color=black] (4.93,0.99) node {$c_{2,2r-1}$};
\fill [color=black] (4,1.73) circle (1.5pt);
\draw[color=black] (4.34,1.86) node {$c_{2,3}$};
\draw[color=red] (3.72,2.22) node {8};
\fill [color=black] (3.5,2.6) circle (1.5pt);
\draw[color=black] (3.84,2.72) node {$c_{2,2}$};
\draw[color=red] (3.22,3.09) node {8};
\draw[color=red] (2.72,3.96) node {8};
\draw[color=red] (-1.91,4.32) node {7};
\fill [color=black] (-1.5,4.33) circle (1.5pt);
\draw[color=black] (-1.08,4.46) node {$c_{2,2r-1}$};
\fill [color=black] (-0.5,4.33) circle (1.5pt);
\draw[color=black] (-0.14,4.46) node {$c_{2,-3}$};
\draw[color=red] (0.09,4.32) node {7};
\fill [color=black] (0.5,4.33) circle (1.5pt);
\draw[color=black] (0.86,4.46) node {$c_{2,-2}$};
\draw[color=red] (1.09,4.32) node {7};
\fill [color=black] (1.5,4.33) circle (1.5pt);
\draw[color=black] (1.86,4.46) node {$c_{2,-1}$};
\draw[color=red] (2.09,4.32) node {7};
\fill [color=black] (2.37,5.32) circle (1.5pt);
\draw[color=black] (2.79,5.44) node {$c_{2,1^*}$};
\fill [color=black] (1.38,5.47) circle (1.5pt);
\draw[color=black] (1.8,5.59) node {$c_{2,2^*}$};
\fill [color=black] (2.05,6.21) circle (1.5pt);
\draw[color=black] (2.46,6.34) node {$c_{2,r^*}$};
\fill [color=black] (3.04,6.36) circle (1.5pt);
\draw[color=black] (3.49,6.49) node {$c_{2,2r-1^*}$};
\fill [color=black] (3.92,3.85) circle (1.5pt);
\draw[color=black] (4.28,3.98) node {$c_{2,1'}$};
\fill [color=black] (4.47,4.69) circle (1.5pt);
\draw[color=black] (4.82,4.81) node {$c_{2,2'}$};
\fill [color=black] (4.36,5.68) circle (1.5pt);
\draw[color=black] (4.69,5.8) node {$c_{2,r'}$};
\fill [color=black] (3.42,5.33) circle (1.5pt);
\draw[color=black] (3.84,5.45) node {$c_{2,2r-1'}$};
\fill [color=black] (2.21,5.81) circle (1.5pt);
\draw[color=black] (2.35,5.93) node {$c_{2,2r^*}$};
\draw[color=black] (3.61,6.06) node {$\vartheta_1$};
\draw[color=black] (1.68,1.39) node {$\vartheta_2$};
\fill [color=black] (3,4.33) circle (1.5pt);
\draw[color=black] (3.21,4.46) node {$O_2$};
\end{scriptsize}
\end{tikzpicture}
	\captionof{figure}{$C_2$ drawn in its other possible $(1,d)$-representation: when $C$ is turned off}
	\label{C8}
\end{minipage}
\begin{minipage}{.49\textwidth}
	\centering
	\usetikzlibrary{arrows}
\definecolor{qqzzqq}{rgb}{0,0.6,0}
\definecolor{qqqqff}{rgb}{0,0,1}
\definecolor{ffqqqq}{rgb}{1,0,0}
\begin{tikzpicture}[line cap=round,line join=round]
\clip(-1.6,-5.81) rectangle (5.39,-1.94);
\draw (0.9,-1.94) node[anchor=north west] {$L_1$ red edges};
\draw (0.78,-5.11) node[anchor=north west] {$L_2$ green edges};
\draw [color=ffqqqq] (-0.53,-3.61)-- (-1.19,-4.1);
\draw [color=ffqqqq] (0.2,-3.24)-- (-0.53,-3.61);
\draw [color=ffqqqq] (0.99,-3.01)-- (0.2,-3.24);
\draw [dash pattern=on 2pt off 2pt,color=ffqqqq] (1.81,-2.93)-- (0.99,-3.01);
\draw [color=ffqqqq] (2.63,-3.01)-- (1.81,-2.93);
\draw [color=ffqqqq] (3.42,-3.24)-- (2.63,-3.01);
\draw [color=ffqqqq] (4.16,-3.61)-- (3.42,-3.24);
\draw [color=ffqqqq] (4.81,-4.1)-- (4.16,-3.61);
\draw [color=qqqqff] (4.81,-4.1)-- (3.81,-4.1);
\draw [shift={(4.31,-3.84)},color=qqzzqq]  plot[domain=2.05:4.23,variable=\t]({0*0.56*cos(\t r)+-1*0.56*sin(\t r)},{1*0.56*cos(\t r)+0*0.56*sin(\t r)});
\draw [shift={(3.31,-3.84)},color=qqzzqq]  plot[domain=2.05:4.23,variable=\t]({0*0.56*cos(\t r)+-1*0.56*sin(\t r)},{1*0.56*cos(\t r)+0*0.56*sin(\t r)});
\draw [shift={(2.31,-3.84)},color=qqzzqq]  plot[domain=2.05:4.23,variable=\t]({0*0.56*cos(\t r)+-1*0.56*sin(\t r)},{1*0.56*cos(\t r)+0*0.56*sin(\t r)});
\draw [shift={(1.31,-3.84)},dash pattern=on 2pt off 2pt,color=qqzzqq]  plot[domain=2.05:4.23,variable=\t]({0*0.56*cos(\t r)+-1*0.56*sin(\t r)},{1*0.56*cos(\t r)+0*0.56*sin(\t r)});
\draw [shift={(0.31,-3.84)},color=qqzzqq]  plot[domain=2.05:4.23,variable=\t]({0*0.56*cos(\t r)+-1*0.56*sin(\t r)},{1*0.56*cos(\t r)+0*0.56*sin(\t r)});
\draw [shift={(-0.69,-3.84)},color=qqzzqq]  plot[domain=2.05:4.23,variable=\t]({0*0.56*cos(\t r)+-1*0.56*sin(\t r)},{1*0.56*cos(\t r)+0*0.56*sin(\t r)});
\draw [shift={(-0.45,-3.5)}] plot[domain=1.57:3.14,variable=\t]({1*0.74*cos(\t r)+0*0.74*sin(\t r)},{0*0.74*cos(\t r)+1*0.74*sin(\t r)});
\draw [shift={(-0.82,-4.53)}] plot[domain=3.14:4.71,variable=\t]({1*0.37*cos(\t r)+0*0.37*sin(\t r)},{0*0.37*cos(\t r)+1*0.37*sin(\t r)});
\draw [shift={(4.44,-4.53)}] plot[domain=-1.57:0,variable=\t]({1*0.37*cos(\t r)+0*0.37*sin(\t r)},{0*0.37*cos(\t r)+1*0.37*sin(\t r)});
\draw [shift={(4.07,-3.5)}] plot[domain=0:1.57,variable=\t]({1*0.74*cos(\t r)+0*0.74*sin(\t r)},{0*0.74*cos(\t r)+1*0.74*sin(\t r)});
\draw [shift={(1.6,-5.14)}] plot[domain=0:1.57,variable=\t]({1*0.24*cos(\t r)+0*0.24*sin(\t r)},{0*0.24*cos(\t r)+1*0.24*sin(\t r)});
\draw [shift={(2.09,-5.14)}] plot[domain=1.57:3.14,variable=\t]({1*0.24*cos(\t r)+0*0.24*sin(\t r)},{0*0.24*cos(\t r)+1*0.24*sin(\t r)});
\draw [shift={(1.6,-2.52)}] plot[domain=-1.57:0,variable=\t]({1*0.24*cos(\t r)+0*0.24*sin(\t r)},{0*0.24*cos(\t r)+1*0.24*sin(\t r)});
\draw [shift={(2.09,-2.52)}] plot[domain=3.14:4.71,variable=\t]({1*0.24*cos(\t r)+0*0.24*sin(\t r)},{0*0.24*cos(\t r)+1*0.24*sin(\t r)});
\draw (-0.45,-2.76)-- (1.6,-2.76);
\draw (2.09,-2.76)-- (4.07,-2.76);
\draw (4.44,-4.9)-- (2.09,-4.9);
\draw (1.6,-4.9)-- (-0.82,-4.9);
\begin{scriptsize}
\fill [color=black] (-1.19,-4.1) circle (1.5pt);
\fill [color=black] (-0.53,-3.61) circle (1.5pt);
\fill [color=black] (-1.19,-4.1) circle (1.5pt);
\fill [color=black] (0.2,-3.24) circle (1.5pt);
\fill [color=black] (0.2,-3.24) circle (1.5pt);
\fill [color=black] (0.99,-3.01) circle (1.5pt);
\fill [color=black] (0.2,-3.24) circle (1.5pt);
\fill [color=black] (1.81,-2.93) circle (1.5pt);
\fill [color=black] (1.81,-2.93) circle (1.5pt);
\fill [color=black] (2.63,-3.01) circle (1.5pt);
\fill [color=black] (1.81,-2.93) circle (1.5pt);
\fill [color=black] (3.42,-3.24) circle (1.5pt);
\fill [color=black] (2.63,-3.01) circle (1.5pt);
\fill [color=black] (4.16,-3.61) circle (1.5pt);
\fill [color=black] (3.42,-3.24) circle (1.5pt);
\fill [color=black] (4.81,-4.1) circle (1.5pt);
\fill [color=black] (4.16,-3.61) circle (1.5pt);
\fill [color=black] (4.81,-4.1) circle (1.5pt);
\fill [color=black] (3.81,-4.1) circle (1.5pt);
\draw[color=red] (4.4,-4.65) node {\large{1}};
\draw[color=red] (3.4,-4.65) node {\large{1}};
\draw[color=red] (2.4,-4.65) node {\large{1}};
\draw[color=red] (0.41,-4.65) node {\large{1}};
\draw[color=red] (-0.59,-4.65) node {\large{1}};
\fill [color=black] (-1.19,-4.1) circle (1.5pt);
\fill [color=black] (-0.19,-4.1) circle (1.5pt);
\fill [color=black] (0.81,-4.1) circle (1.5pt);
\fill [color=black] (1.81,-4.1) circle (1.5pt);
\fill [color=black] (2.81,-4.1) circle (1.5pt);
\fill [color=black] (3.81,-4.1) circle (1.5pt);
\fill [color=black] (4.81,-4.1) circle (1.5pt);
\end{scriptsize}
\end{tikzpicture}
	\captionof{figure}{Graph $D$}
	\label{D3}
\end{minipage}
\end{figure}

\begin{lemma}
In every $(1,d)$-representation of $C$, at least one of vectors $\vartheta_1$ and $\vartheta_2$ has length $1$. Also, for any pair of vectors $\left(w_1,w_2\right)$ with
\[
\left(\left(\left\lvert w_1\right\rvert=1\right)\land\left(\left\lvert w_2\right\rvert\le\frac{U}{L}\right)\right)\lor\left(\left(\left\lvert w_1\right\rvert\le\frac{U}{L}\right)\land\left(\left\lvert w_2\right\rvert=1\right)\right),
\]
there exists a $(1,d)$-representation of $C$ in which $\vartheta_1=w_1$ and $\vartheta_2=w_2$.
\end{lemma}

\begin{proof}
For $C_1$, the drawn one is the only $(1,d)$-representation up to isometry, since rhombus $c_{1,0}c_{1,1}c_{1,1^*}c_{1,-1}$ is rigid and with the green edges numbered by $1$, $2$ and $3$, we can get to all the vertices of $C_1$.

This means that in any $(1,d)$-representation of $C_2$, the locations of $c_{2,-2r-1^*}$, $c_{2,-2r'}$, $c_{2,2r}$, $c_{2,2r'}$, $c_{2,2r+1^*}$ and $c_{2,-2r}$ are fixed up to isometry and since the distance of $c_{2,-2r}$ and $c_{2,2r}$ is $2r\cdot\sqrt{3}$, the only ways the path $c_{2,i}$ $(-2r\le i\le 2r)$ can be drawn are the two drawn ones. Call the first possibility the case, when $C$ is {\it turned on} and the second possibility the case, when $C$ is {\it turned off}. If $C$ is turned on, then $\left\lvert c_{2,0}c_{2,2r+1^*}\right\rvert=2r+1$ and $\left\lvert c_{2,1}c_{2,2r'}\right\rvert=2r$, thus paths $c_{2,0}c_{2,1^*}c_{2,2^*}...c_{2,2r+1^*}$ and $c_{2,1}c_{2,1^*}c_{2,2^*}...c_{2,2r'}$ only have the $(1,d)$-representation drawn in Figure \ref{C7} in which $\left\lvert\vartheta_1\right\rvert=1$.

If $C$ is turned off, then $\left\lvert c_{2,0}c_{2,-2r-1^*}\right\rvert=2r+1$ and $\left\lvert c_{2,0}c_{2,-2r-1'}\right\rvert=2r$, thus paths $c_{2,0}c_{2,-1^*}c_{2,-2^*}...c_{2,-2r-1^*}$ and $c_{2,0}c_{2,-1'}c_{2,-2'}...c_{-2r'}$ only have the $(1,d)$-representation drawn in Figure \ref{C8} in which $\left\lvert\vartheta_2\right\rvert=1$. In the first case, since all four endpoints of paths $c_{2,0}c_{2,-1^*}c_{2,-2^*}...c_{2,-2r-1^*}$ and $c_{2,1}c_{1,-1'}c_{1,-2'}...c_{2,-2r'}$ have distance less than $0.9$ from $O_1$, we can apply Lemma \ref{lem:path} to $c_{2,0}c_{2,-1^*}c_{2,-2^*}...c_{2,-r^*}$, $c_{2,-r^*}c_{2,-r-1^*}...c_{-2r-1^*}$, $c_{2,1}c_{2,-1'}c_{2,-2'}...c_{2,-r'}$ and $c_{2,-r'},c_{2,-r-1'},...,c_{2,-2r'}$ and prove that except for the finitely many other vertices of $C_2$, both $c_{2,-r^*}$ and $c_{2,-r'}$ can be located anywhere inside the disk of radius $r-0.9>\left\lceil\frac{U}{2L}\right\rceil$ around $O_1$, and we can get a $(1,d)$-representation of $C$ which fulfills these locations. And thus, $\vartheta_2$ can be any vector with length at most $\frac{U}{L}$. Similarly, if $C$ is turned off, $\vartheta_1$ is fixed, but $\vartheta_2$ can be any vector with length at most $\frac{U}{L}$.
\end{proof}

\begin{lemma}
$ran(D)=\left(0,L\right]$.
\end{lemma}

\begin{proof}
The path formed by the green edges clearly has length $L_2\cdot d$, while the path formed by the red edges can move arbitrarily in the upper half plane above the two endpoints of $D$, so $d$ can take any value in $\left(0,L\right)$. And since $L_1$ and $L_2$ are relative primes, even $d=L$ cannot cause a problem.
\end{proof}

Now create a PVEBG $\hat{G}_L^U$ in the following way:

Call the red edges of $G$ $e_1,e_2,...,e_{\left\lvert E_r(G)\right\rvert}$ and call the red edges of $D$ $e_1',e_2',...,e_{L_1}'$. Now create $L_1\cdot\left\lvert E_r(G)\right\rvert$ copies of $C$ and call them $C^{(i,j)}$ ($1\le i\le\left\lvert E_r(G)\right\rvert$, $1\le j\le L_1$). Also, call the copies of $\vartheta_1$ and $\vartheta_2$ $\vartheta_1^{(i,j)}$ and $\vartheta_2^{(i,j)}$. Now delete all the $e_i$ ($1\le i\le \left\lvert E_r(G)\right\rvert$) and instead connect the two endpoints of $e_i$ and the two endpoints of $\vartheta_1^{(i,j)}$ for $1\le i\le\left\lvert E_r(G)\right\rvert$ and $1\le j\le L_1$. Similarly, delete all the $e_j'$ ($1\le i\le L_1$) and instead connect the two endpoints of $e_j'$ with $\vartheta_2^{(i,j)}$ for $1\le i\le\left\lvert E_r(G)\right\rvert$ and $1\le j\le L_1$. Also, refer to the subgraph spanned by the vertices of $G$ $G'$ and the subgraph spanned by the vertices of $D$ $D'$. And if we only take those $(1,d)$-representations of $\hat{G}_L^U$, for which $d\le U$, then using Proposition \ref{prop:pvebgtoebg}, we can create a satisfactory EBG $G_L^U$.

\begin{lemma}
$ran\left(\hat{G}_L^U\right)\cap\left(L,+\infty\right)\setminus ran(G)=\emptyset$.
\end{lemma}

\begin{proof}
Suppose that for some $d$, $\hat{G}_L^U$ has a $(1,d)$-representation. If for any $i$ and any $j$, $C^{(i,j)}$ is turned on, then the two endpoints of $e_i$ must have distance $1$ because of Lemma \ref{lem:grid}. Similarly, if it is turned off, then the two endpoints of $e_j'$ must have distance $1$. Now suppose that there exists an $i$, for which all the $C^{(i,j)}$ are turned off, it would mean that for all the $j$, at least one $C^{(i,j)}$ is turned off, thus for all $j$, the endpoints of $e_j'$ have distance $1$. If such an $i$ does not exist, then for all $i$, the endpoints of $e_i$ have distance $1$. So either $G^*$ forms a $(1,d)$-representation of $G$ or $D^*$ forms a $(1,d)$-representation of $D$. Thus, $ran(\hat{G}_L^U)\subseteq ran(G)\cup ran(D)$, which proves the statement.
\end{proof}

\begin{lemma}
$ran\left(\hat{G}_L^U\right)\supseteq\left(0,L\right]\cup ran(G)\setminus\left(U,+\infty\right)$.
\end{lemma}

\begin{proof}
Call a representation of a graph similar to a $(1,d)$-representation but in which the red edges go to segments of length $x$ and the blue edges go to segments of length $y$ an $(x,y)$-representation. Also, call any $(x',y)$-representation with $x'\le x$ an $(x',y)$-representation.

From Proposition \ref{prop:gridconnection} we know that for any $(1,d)$-representation of $G$ and $(\le\frac{U}{L},d)$-representation of $D$, there exists a $(1,d)$-representation of $\hat{G}_L^U$ (in which all the $C^{(i,j)}$ are turned on). And since $D$ always has such a representation for $\left(0,U\right]$, this proves $ran\left(\hat{G}_L^U\right)\subseteq ran(G)\setminus\left(U,+\infty\right)$

We also know that for any $(1,d)$-representation of $D$ and $(\le\frac{U}{L},d)$-representation of $G$, there exists a $(1,d)$-representation of $\hat{G}_L^U$ (in which all the $C^{(i,j)}$ are turned off). And since $G$ always has such a representation for $d\le L$ (where $ran(G)\cap\left(L,+\infty\right)\neq\emptyset$), this proves $ran(G)\supseteq\left(0,L\right]$.
\end{proof}
\end{proof}

\hypertarget{pf:valami2}{}\subsection{Proof of Proposition \ref{prop:valami2}}

First, we reverse the colouring of the edges of $G$ and get a graph we will call $G^{*}$. Because of Lemma \ref{lem:inverse}, $ran\left(G^{*}\right)\setminus0=\left\lbrace d\vert\frac{1}{d}\in ran(G)\right\rbrace$. Now construct $\left(G^{*}\right)_{\frac{1}{U_a}}^{\frac{1}{L_b}}$ using Proposition \ref{prop:valami}. Now we reverse the colouring of the edges of $\left(G^{*}\right)_{\frac{1}{U_a}}^{\frac{1}{L_b}}$ and again from Lemma \ref{lem:inverse}, we get that for this graph, $ran\left(\left(\left(G^{*}\right)_{\frac{1}{U_a}}^{\frac{1}{L_b}}\right)^{*}\right)\supseteq\left[U_a,+\infty\right)\cup ran(G)\setminus\left[0,L_b\right)$ and $ran\left(\left(\left(G^{*}\right)_{\frac{1}{U_a}}^{\frac{1}{L_b}}\right)^{*}\right)\cap\left(0,U_a\right)\setminus ran(G)=\emptyset$. Now we apply Proposition \ref{prop:valami} again for this graph, but now with $L=L_a$ and $U=U_b$. The graph we get in the end suffices for $G_{L_a,L_b}^{U_a,U_b}$, since its range contains $\left[L_b,L_a\right]\cup\left[U_a,U_b\right]\cup ran(G)\cap\left(L_a,U_a\right)$.\qed

\section{Concluding remarks and open problems}

We mention a result of Maehera \cite{m}, which states that a number can be realized as the distance of two vertices of a rigid unit distance graph exactly if it is algebraic.

\begin{problem}
Can we construct edge-bicoloured graphs for all semialgebraic sets in $\mathbb{R}$ without the boundedness limitations?
\end{problem}

\begin{problem}\label{prob:1dgraphdistinctvertices}
What if the points representing different vertices do not have to be distinct?
\end{problem}

Note that the constructions used in the proof of Theorem \ref{thm:1dsemialgebraic} do not work for any of the variants above: one of the reasons is that these constructions heavily rely on the fact that in any graph, one can assure that a pair of vertices has the same vector as another pair of vertices by connecting these pairs by a grid (see Figure \ref{fig:redrhombusgrid}).

\begin{problem}
What are the possible ranges if the colouring of the edges is not given in advance?
\end{problem}

\begin{problem}
Can we generalize the result in some other way (more than two distances, more than two dimensions)?
\end{problem}

\section{Acknowledgement}

I would like to thank Dömötör Pálvölgyi for posing the problem and for all the help he gave me in the past years.

\end{document}